\documentclass[11pt]{article}
\usepackage{latexsym}
\usepackage{amsmath,amssymb,epsfig,subfigure}
\usepackage{amsthm,enumerate,verbatim}
\usepackage{amsfonts}
\usepackage{color}
\usepackage{amsmath}
\usepackage{graphicx}
\usepackage{epstopdf}
\usepackage{indentfirst}

\providecommand{\U}[1]{\protect\rule{.1in}{.1in}}
\providecommand{\U}[1]{\protect\rule{.1in}{.1in}}
\newcommand{\BE}{\begin{equation}}
\newcommand{\EE}{\end{equation}}

\numberwithin{equation}{section}
\newtheorem{proposition}{Proposition}[section]
\newtheorem{theorem}[proposition]{Theorem}
\newtheorem{lemma}[proposition]{Lemma}

\newtheorem{remark}[proposition]{Remark}

\def\dfrac{\displaystyle\frac}
\def\dsum{\displaystyle\sum}

\begin{document}

 \title{\textbf{A fast linearized finite difference method for the nonlinear multi-term time-fractional wave equation}}

 \author{Pin Lyu\thanks{Email: plyu@swufe.edu.cn. School of Economic Mathematics, Southwestern University of Finance and Economics, Chengdu, China. }
 \and Yuxiang Liang \thanks{Email: 13026876611@163.com. School of Applied Mathematics, Guangdong University of Technology, Guangzhou 510006, Guangdong, China.}
 \and Zhibo Wang\thanks{Corresponding author. Email: wzbmath@gdut.edu.cn. School of Applied Mathematics, Guangdong University of Technology,
Guangzhou 510006, Guangdong, China.}}
\date{}
 \maketitle\normalsize

 \begin{abstract}
 In this paper, we study a fast and linearized finite difference method to solve the nonlinear time-fractional wave equation with multi fractional orders. We first propose a discretization to the multi-term Caputo derivative based on the recently established fast ${\cal L}2$-$1_\sigma$ formula and a weighted approach. Then we apply the discretization to construct a fully fast linearized discrete scheme for the nonlinear problem under consideration. The nonlinear term, which just fulfills the Lipschitz condition, will be evaluated on the previous time level. Therefore only linear systems are needed to be solved for obtaining numerical solutions. The proposed scheme is shown to have second-order unconditional convergence with respect to the discrete $H^1$-norm. Numerical examples are provided to justify the efficiency.
 \end{abstract}

 {{\bf Key words:}
 nonlinear fractional equation; multi-term derivative; fast algorithm; linearized method; second-order scheme
 }


\section{Introduction}

Fractional derivatives were found to be more accurate tools to describe diverse materials and processes which present memory and hereditary properties, thus it gives rise to great interest in the study of fractional differential equations (FDEs). However, the closed-form solution of most of FDEs are hardly to be obtained.  Investigating the efficient numerical methods for FDEs becomes a popular and also urgent topic.

Recently, the multi-term time-fractional diffusion and wave equations, which is the single fractional order in the FDEs generalized to multi-orders \cite{Hesameddini,Podlubny,multitermsub-d}, were successfully applied to model various types of visco-elastic damping \cite{Shiralashetti} and to describe the phenomenon of subdiffusion of oxygen in both transverse and longitudinal directions \cite{Srivastava}.
Jin et al. studied the multi-term fractional diffusion equation by using the Galerkin finite element method where its extraordinary capability of modeling anomalous diffusion phenomena in highly heterogeneous aquifers and complex viscoelastic materials were mentioned \cite{Galerkin finite element}.
The theoretical and numerical methods to the solution of linear multi-term time-fractional equations were studied by many researchers \cite{Daftardar2008,Dehgha2015,LiuF2013,Luchko2011}. But the studies of numerical solutions to the nonlinear multi-term time-fractional wave equations are still limited.

In this paper, we consider an efficient numerical method to the nonlinear multi-term time-fractional wave equation:
\begin{align}\label{eq1}
&\sum_{r=0}^m\lambda_r~ {_0^CD}_t^{\alpha_r} u=\frac{\partial^2 u}{\partial x^2}+f(u)+p(x,t),\quad x\in\Omega,~t\in(0,T],\\\label{eq2}
&u(x,t)=0,\quad x\in\partial\Omega,\quad t\in(0,T],\\\label{eq3}
&u(x,0)=\varphi(x),~u_t(x,0)=\psi(x),\quad x\in\bar\Omega,
\end{align}
where $\Omega=(x_L,x_R)$, $1<\alpha_m<\cdots<\alpha_0<2$, and the positive weights $\lambda_r$ fulfills
$\dsum_{r=0}^m\lambda_r\leq C$. Here and hereafter, $C$ always denotes a generic positive constant.
Moreover, ${_0^CD}_t^{\delta}$ denotes the Caputo derivative of order $\delta$:
$${_0^CD}_t^{\delta}u(t)=\frac1{\Gamma(n-\delta)}\int_{0}^t\frac{u^{(n)}(s)}{(t-s)^{\delta+1-n}}ds,\quad n=\lceil{\delta}\rceil.$$
The nonlinear function $f$ in \eqref{eq1} satisfies the Lipschitz condition:
$$|f(\phi_1)-f(\phi_2)|\leq L|\phi_1-\phi_2|,~~\forall\phi_1,~\phi_2\in\Omega_f,$$
where $\Omega_f$ is a suitable domain, and $L$ is a positive constant only depends on the domain $\Omega_f.$

The nonlocal dependence of fractional derivative is inherited by its discretizations.  Thus it should be necessary to develop high accurate and/or fast algorithms to FDEs in order to save the computational costs and the memory storage. Moreover, the  problem under consideration is a nonlinear time-fractional equation, and  classical nonlocal numerical schemes blending with iterative methods would take more expensive computation and have more complexity in the analysis. So it is really competitive and also necessary to construct linearized numerical methods for nonlinear time-fractional problems \cite{LiaoYanZhang:2018,LyuVongKGS,LyuVong-AML2018,LyuVongNA2017,VongLyu-JSC2018}.
 It is worth to mention that the solution of many time-fractional differential equations typically displays a weak singularity near the initial time \cite{Jin-IMA2016,Jin-SIAM2016,Liao-SIAM2018,McLeanMustapha:2007,MustaphaMcLean:2013,Mustapha-IMA2014,SakamotoYamamoto:2011,Stynes-SIAM2017,Stynes-FCAA2016}, which leads to the loss of time accuracy  for many related high-order numerical methods.
 The nonuniform grids technique \cite{Liao-SIAM2018,Stynes-SIAM2017} is a very popular method to recover the full accuracy very recently.
 Since our proposed linearized method for the considered nonlinear time-fractional problem is based on a weighted approach,  the analysis will be very difficult if the grids are nonuniform.  In view of the facts, based on the fast ${\cal L}2$-$1_\sigma$ formula (named ${\cal FL}2$-$1_\sigma$) investigated in \cite{FL21}, we construct a fast linearized finite difference scheme on uniform grids to solve the nonlinear multi-term time fractional wave equation. This may offset the possible accuracy loss in the computational sense.

Traditional direct numerical methods for time-fractional PDEs require ${\cal O}(MN)$ memory and ${\cal O}(MN^2)$ work,  where $M$ and $N$ denote the total number of space steps and time steps, respectively.
Based on an efficient sum-of-exponentials (SOE) approximation for the kernel function $t^{-\beta-1}$ on the interval $[\tau,T]$, where $\beta\in(0,1)$ and $\tau$ is the time step size, and combined with the classical ${\cal L}1$ discretization \cite{LinXu2007,Oldham1974,SunWu:2006}, Jiang et al. \cite{FL1} proposed an efficient fast evaluation to the Caputo derivative, and then applied it to solve time-fractional diffusion equations.
The corresponding fast numerical algorithm requires only ${\cal O}(MN_{\exp})$ memory and ${\cal O}(MNN_{\exp})$ work, where $N_{\exp}$ is the number of exponentials and of order ${\cal O}(\log N)$.  Based on the ${\cal L}2$-$1_\sigma$ discretization \cite{AAA}, Yan et al. \cite{FL21} then constructed a fast and second-order ${\cal FL}2$-$1_\sigma$  numerical method for time-fractional diffusion equations. The computational cost just require ${\cal O}(MN\log ^2N)$ and the overall storage is ${\cal O}(M\log ^2N)$.

In this paper, based on the ${\cal FL}2$-$1_\sigma$ formula in \cite{FL21} and the weighted idea in \cite{LyuVongNA2017}, we first approximate the multi-term derivative in \eqref{eq1} at time point $t_n$ ($n\geq1$) by a weighted discretization which will solve the function to time level $t_{n+1}$, and give a fitted time approximation on the diffusion term.
Then we apply the proposed approximation to construct a linearized and fast finite difference scheme to solve the nonlinear equation \eqref{eq1}--\eqref{eq3}, which will evaluate the nonlinear term at previous time level.
Thus only linear systems are needed to be solved for obtaining approximated solutions.
With some important rigorously verified  properties of the  discrete coefficients, we show that our proposed fast linearized numerical scheme is unconditionally convergent with the order ${\cal O}(h^2+\tau^2+\epsilon)$, where $h$ is the spatial step size and $\epsilon$ is the tolerance error in the approximation of SOE to the kernel.

The structure of the paper is as follows. In section \ref{Sec-approximation}, based on the ${\cal FL}2$-$1_\sigma$ formula and a weighted approach, we propose a discretization to approximate the multi-term Caputo derivative at  time grid point $t_{n}$ ($n\geq 1$). With some necessary properties of the corresponding coefficients being verified, we give the truncation error analysis for the proposed discretization.
Then applying the discretization and using a special approximation for first time level solution, we construct a fully fast-linearized finite difference scheme to solve the considered problem \eqref{eq1}--\eqref{eq3}. In section \ref{Sec-conver}, we first estimate some important properties of the discrete coefficients, and then we show that our proposed fast-linearized scheme is unconditionally convergent with second-oder accuracy with respect to discrete $H^1$-norm.
In section \ref{Sec-numer}, numerical examples are carried out to confirm the efficiency of the numerical scheme. A brief conclusion is followed in section \ref{conclusion}.

\section{The fast linearized numerical method}\label{Sec-approximation}
Notations for clarifying some coefficients and parameters:
\begin{itemize}
\item  $g_k^{(n+1,\beta)}$ ----- the ${\cal L}2$-$1_\sigma$ type coefficients;
\item  ${\widehat g}_{k}^{(n+1)}$ ----- the ${\cal L}2$-$1_\sigma$ type coefficients of multi-term Caputo derivative;
\item  $N^{(\beta)},s_j^{(\beta)},\omega_j^{(\beta)},{\widehat \omega}_j^{(\beta)}$ ----- the corresponding parameters of SOE approximation;
\item  $^{\cal F}g_k^{(n+1,\beta)}$ ----- the ${\cal FL}2$-$1_\sigma$ type coefficients;
\item  $^{\cal F}{\widehat g}_{k}^{(n+1)}$ ----- the ${\cal FL}2$-$1_\sigma$ type coefficients of multi-term Caputo derivative;
\item  ${\bf g}_k^{(n+1)}$ ----- the refined ${\cal FL}2$-$1_\sigma$ type coefficients of multi-term Caputo derivative.
\end{itemize}
\subsection{Preliminary}
We first introduce some temporal notations. For a given positive integer $N$, denote the uniform time step size by $\tau=\frac TN$, and take $t_n=n\tau~(0\leq n\leq N),~t_{n+\sigma}~(\sigma\in[0,1])$. Let  ${\cal W}_\tau=\big\{w^n|0\leq n\leq N\big\}$.
For any $u^n\in{\cal W_\tau},$ denote
$$\delta_tu^{n+\frac12}=\frac{u^{n+1}-u^n}\tau,~~\delta_{\hat t}u^n=\frac{u^{n+1}-u^{n-1}}{2\tau}.$$
For simplicity, we denote the multi-term derivative by
$${_\lambda}{D}_t^\alpha u(t)=\sum_{r=0}^m\lambda_r~ {_0^CD}_t^{\alpha_r}u(t).$$
In the rest of the paper, we take $v=u_t$, $\alpha_r=\beta_r+1$ or $\alpha=\beta+1$ when it is required. Then
\begin{align}\label{relation-Dv-Du}
{_\lambda D}_t^\alpha u(t)={_\lambda D}_t^\beta v(t).
\end{align}

Our approximation to fractional derivative is based on the ${\cal L}2$-$1_\sigma$ discretization in \cite[Lemma 2]{AAA},
we first introduce the corresponding coefficients. For $\beta\in(0,1)$, denote
\begin{align*}
&a_0^{(\beta)}=\sigma^{1-\beta},\quad  a_l^{(\beta)}=(l+\sigma)^{1-\beta}-(l-1+\sigma)^{1-\beta},~\mu^{(\beta)}=\frac{\tau^{-\beta}}{\Gamma(2-\beta)},\\
&b_l^{(\beta)}=\frac1{2-\beta}[(l+\sigma)^{2-\beta}-(l-1+\sigma)^{2-\beta}]-\frac12[(l+\sigma)^{1-\beta}+(l-1+\sigma)^{1-\beta}], \quad l\geq 1,
\end{align*}
where the particular $\sigma$ (here and hereafter) is generated in Lemma \ref{trunerror2}, which may be different from the $\sigma$ in \cite{AAA}.
Further denote $g_0^{(1,\beta)}=\mu^{(\beta)}a_0^{(\beta)}$ and
\begin{equation}\notag
g_k^{(n+1,\beta)}=\mu^{(\beta)}\cdot\left\{\begin{array}{ll}
a_n^{(\beta)}-b_n^{(\beta)},\quad &k=0,\\
a_{n-k}^{(\beta)}+b_{n-k+1}^{(\beta)}-b_{n-k}^{(\beta)},\quad &1\leq k\leq n-1,\\
a_0^{(\beta)}+b_1^{(\beta)},\quad &k=n,
\end{array}\right.\quad n\geq 1.
\end{equation}
The high-order approximation proposed in \cite{multiterm} to multi-term Caputo derivative is:
\begin{align*}
{\cal ^H {\widehat D}}_t^\beta v(t_{n+\sigma})\triangleq&\sum_{r=0}^m\frac{\lambda_r}{\Gamma(1-\beta_r)}\bigg[\sum_{k=1}^{n}\int_{t_{k-1}}^{t_k}(\Pi_{2,k}v(s))'(t_{n+\sigma}-s)^{-\beta_r} ds\\
&+\int_{t_{n}}^{t_{n+\sigma}}(\Pi_{1,n}v(s))'(t_{n-1+\sigma}-s)^{-\beta_r} ds \bigg]\\
=&\sum_{k=0}^{n} {\widehat g}_{k}^{(n+1)}[v(t_{k+1})-v(t_{k})],
\end{align*}
where $\Pi_{2,k}$ and $\Pi_{1,n}$ are the quadratic and linear polynomials \cite{multiterm}, respectively; and
$${\widehat g}_{k}^{(n+1)}=\sum_{r=0}^m\lambda_rg_k^{(n+1,\beta_r)}.$$
The next lemma is given with respect to the truncation error between ${_\lambda D}_t^\beta v(t_{n+\sigma})$ and ${\cal ^H{\widehat D}}_t^\beta v(t_{n+\sigma})$, where the basic way to approximate the multi-term derivative is shown.

\begin{lemma} {\rm(\cite{multiterm})}\label{trunerror2}
Suppose $v(t)\in{\cal C}^3[0,t_{n+1}]$, it holds that
$$\left|{_\lambda D}_t^\beta v(t_{n+\sigma}) -{\cal ^H{\widehat D}}_t^\beta v(t_{n+\sigma})\right|
\leq C\sum_{r=0}^m\lambda_r \tau^{3-\beta_r},$$
 where $\sigma\in \big[1-\frac{\beta_0}{2},1-\frac{\beta_m}{2}\big]$ is the root of equation
 $$F(\sigma)=\dsum_{r=0}^m\frac{\lambda_r \sigma^{1-\beta_r}}{\Gamma(3-\beta_r)}\Big[\sigma-(s-\dfrac{\beta_r}2)\Big]\tau^{2-\beta_r},$$
 generated by the method of Newton iteration.
 \end{lemma}

To construct fast numerical scheme, we introduce a ${\cal FL}2$-$1_\sigma$ formula,
proposed in \cite{FL21}, which is based on the ${\cal L}2$-$1_\sigma$ formula  and a SOE approximation to the kernel function $t^{-\beta}$ in the Caputo derivative. The SOE approximation reads as:
\begin{lemma}\label{soe}{\rm(\cite{FL21})}
For any $\beta\in(0,1)$, tolerance error $\epsilon$, cut-off time step size ${\widehat \tau}$ and final time $T$,
there are some positive integer $N^{(\beta)}$, positive points $s_j^{(\beta)}$ and corresponding positive weights $\omega_j^{(\beta)}, ~(j=1,2,\ldots,N^{(\beta)})$ satisfying
$$\left| t^{-\beta}-\sum_{j=1}^{N^{(\beta)}}\omega_j^{(\beta)} e^{-s_j^{(\beta)} t} \right|\leq \epsilon, \quad \forall t\in[{\widehat \tau},T] ,$$
and the number of exponentials needed is of the order
$$N^{(\beta)}={\cal O}\left(\log \frac1{\epsilon}\left(\log\log \frac1{\epsilon}+\log \frac{T}{\widehat \tau} \right) +\log \frac1{\widehat \tau}\left( \log\log \frac1{\epsilon}+\log \frac1{\widehat \tau} \right)\right).$$
\end{lemma}
Taking ${\widehat \tau}\le10^{-1}$, then we have $N^{(\beta)}={\cal O}\Big(\log\dfrac1{\widehat \tau}\log\dfrac T{\widehat \tau}+\log\dfrac1{\widehat \tau}\log\dfrac 1{\widehat \tau}\Big).$ Then let ${\widehat \tau}=\sigma\tau$ and $\tau<\min\{\sigma,\max\{\sigma^2,T\sigma^2\}\}$, we have
$N^{(\beta)}={\cal O}\Big(\log\dfrac1\tau\log\dfrac T\tau+\log\dfrac1\tau\log\dfrac1\tau\Big).$

For a single-term Caputo derivative of order $\beta\in(0,1)$ at time $t_{n+\sigma}$: ${_{0}^C}D_t^{\beta}v(t_{n+\sigma})$, the ${\cal FL}2$-$1_\sigma$ type formula here is estimating ${(t_{n+\sigma}-s)^{-\beta}}$ $(s\in(0,t_n))$ by Lemma \ref{soe} and $v'(s)$ $(s\in(t_n,t_{n+\sigma}))$ by linear polynomial, and then estimating the history part of the integral by using a recursive relation and quadratic interpolation, that is
\begin{align}\notag
 {_{0}^C}D_t^{\beta}v(t_{n+\sigma})&\approx
 \int_0^{t_n}v'(s)\sum\limits_{j=1}^{N^{(\beta)}}\widehat\omega_j^{(\beta)}e^{-s_j^{(\beta)}(t_{n+\sigma}-s)}ds
 +\dfrac1{\Gamma(1-\beta)}\cdot\dfrac{v^{n+1}-v^n}{\tau}
 \int_{t_n}^{t_{n+\sigma}}\frac1{(t_{n+\sigma}-s)^{\beta}}ds\\
 &\approx\sum\limits_{j=1}^{N^{(\beta)}}\widehat\omega_jV_j^{(n,\beta)}+\mu^{(\beta)} a_0^{(\beta)}(v^{n+1}-v^n)
 \triangleq{^{\cal FH}}{\cal D}_t^\beta v^{n+\sigma},\label{Fast-dv}
 \end{align}
where $${\widehat \omega}_j^{(\beta)}=\frac{\omega_j^{(\beta)}}{\Gamma(1-\beta)},
 ~~V_j^{(n,\beta)}=e^{-s_j^{(\beta)}\tau}V_j^{(n-1,\beta)}+A_j^{(\beta)}(v^{n}-v^{n-1})+B_j^{(\beta)}(v^{n+1}-v^n),$$
 with
 $A_j^{(\beta)}=\int_0^1(\frac32-s)e^{-s_j^{(\beta)}\tau(\sigma+1-s)}ds$,
 $B_j^{(\beta)}=\int_0^1(s-\frac12)e^{-s_j^{(\beta)}\tau(\sigma+1-s)}ds$, and
 $V_j^0=0$.

In fact, the ${\cal FL}2$-$1_\sigma$ type discretization can be rewritten as (see \cite{FL21})
 \begin{align*}
 &{^{\cal FH}}{\cal D}_t^\beta v^{n+\sigma}
 =\sum_{k=0}^n {^{\cal F}g_k^{(n+1,\beta)}}(v^{k+1}-v^k) ,
\end{align*}
where $^{\cal F}g_0^{(1,\beta)}=\mu^{(\beta)} a_0^{(\beta)}$ and 
\begin{equation}\label{fast-coefficient}
^{\cal F}g_k^{(n+1,\beta)}=\left\{\begin{array}{ll}
\dsum_{j=1}^{N^{(\beta)}}{\widehat \omega_j^{(\beta)}}e^{-(n-1)s_j^{(\beta)}\tau}A_j^{(\beta)}, &k=0;\\
\dsum_{j=1}^{N^{(\beta)}}{\widehat \omega_j^{(\beta)}} \Big[e^{-(n-k-1)s_j^{(\beta)}\tau}A_j^{(\beta)}+e^{-(n-k)s_j^{(\beta)}\tau}B_j^{(\beta)}\Big],~~&1\leq k\leq n-1;\\
\dsum_{j=1}^{N^{(\beta)}}{\widehat \omega_j^{(\beta)}}B_j^{(\beta)}+\mu^{(\beta)} a_0^{(\beta)}, &k=n;
\end{array}\right.\quad n\geq 1.
\end{equation}
%
 Denote ${^{\cal FH}}{\cal \widehat D}_t^\beta v^{n+\sigma}=\dsum_{r=0}^m\lambda_r{^{\cal FH}}{\cal D}_t^{\beta_r} v^{n+\sigma}$ and
  \begin{align}\label{gF}
 {^{\cal F} {\widehat g}_k^{(n+1)}}=\sum_{r=0}^m\lambda_r{^{\cal F}g_k^{(n+1,\beta_r)}},
 \end{align}
 then the multi-term ${\cal FL}2$-$1_\sigma$ type discritization can be rewritten as
 \begin{align}\label{for-v1}
 &{^{\cal FH}}{\cal {\widehat D}}_t^\beta v^{n+\sigma}
 =\sum_{k=0}^n {^{\cal F} {\widehat g}_k^{n+1}}(v^{k+1}-v^k) , \quad 0\leq n \leq N-1.
\end{align}
We have the following lemma which presents the error of the above fast approximation to multi-term Caputo derivative.
 \begin{lemma}\label{appro-Dv}
 For $v(t)\in{\cal C}^3[0,t_n]$, it holds that
 \begin{align*}
 &{_\lambda D}_t^\beta v(t_{n+\sigma}) ={^{\cal FH}{\cal {\widehat D}}}_t^\beta v^{n+\sigma}+ {\cal O}(\tau^{3-\beta_0}+\epsilon),~1\leq n\leq N,\\
  &{_\lambda D}_t^\beta v(t_{\sigma}) ={^{\cal FH}{\cal {\widehat D}}}_t^\beta v^{\sigma}+ {\cal O}(\tau^{3-\beta_0}).
  \end{align*}
 \end{lemma}
 \begin{proof}
 For $1\leq n\leq N$, it follows from Lemmas \ref{trunerror2} and \ref{soe} that
 \begin{align*}
 {_\lambda D}_t^\beta v(t_{n+\sigma})
 =& {^{\cal H}}{\cal {\widehat D}}_t^\beta v^{n+\sigma}+{\cal O}(\tau^{3-\beta_0})\\
 =& {^{\cal FH}}{\cal {\widehat D}}_t^\beta v^{n+\sigma}+\sum_{r=0}^m\frac{\lambda_r}{\Gamma(1-\beta_r)}
 \Bigg\{\sum_{k=1}^{n-1}\int_{t_{k-1}}^{t_k}(\Pi_{2,k}v(s))'\bigg[(t_{n-1+\sigma}-s)^{-\beta_r}\\
 & -\sum\limits_{j=1}^{N^{(\beta_r)}}\omega_j^{(\beta_r)}e^{-s_j^{(\beta_r)}(t_{n+\sigma}-s)} \bigg]ds\Bigg\}
  +{\cal O}(\tau^{3-\beta_0})\\
  =&{^{\cal FH}}{\cal {\widehat D}}_t^\beta v^{n+\sigma}+\sum_{r=0}^m\frac{\lambda_r t_n{\cal O}(\epsilon_r)}{\Gamma(1-\beta_r)}\max_{0\leq t\leq t_n}|v'(t)|
  +{\cal O}(\tau^{3-\beta_0})\\
   =&{^{\cal FH}}{\cal {\widehat D}}_t^\beta v^{n+\sigma}
  +{\cal O}(\tau^{3-\beta_0}+\epsilon),
 \end{align*}
where $\epsilon_r$ is the tolerance related to the SOE approximation of the kernel $(t_{n-1+\sigma}-s)^{-\beta_r}$, and $\epsilon=\max\limits_{0\leq r\leq m}\{\epsilon_r\}$.
 As ${^{\cal FH}{\cal {\widehat D}}}_t^\beta v^{\sigma}$ $=$ ${^{\cal H}{\cal {\widehat D}}}_t^\beta v^{\sigma}$, the second conclusion can be directly obtained from Lemma \ref{trunerror2}.
 \end{proof}

 \subsection{Fast approximation to multi-term derivative}
 \subsubsection{A weighted method to approximate ${_\lambda D}_t^\alpha u(t_{n})$}
Now we go to study a second-order approximation to the multi-term derivative ${_\lambda D}_t^\alpha u(t_{n})$ based on the method of order reduction and a weighted approach.
The following lemma provides an important weighted approach for discretizing ${_\lambda D}_t^\alpha u(t_{n})~(1\leq n\leq N-1)$ that leads to the linearized approximation to our considered nonlinear problem.
\begin{lemma}{\rm (\cite{LyuVongNA2017})}\label{weigted-app}
For any $f(t)\in{\cal C}^2[t_{n-1+\sigma},t_{n+\sigma}]$, it holds that
$$f(t_n)=(1-\sigma)f(t_{n+\sigma})+\sigma f(t_{n-1+\sigma})+{\cal O}(\tau^2).$$
\end{lemma}
Before applying Lemma \ref{weigted-app}, we first observe that \eqref{fast-coefficient} and \eqref{gF}
can yield ($n\geq1$)
\begin{align}\label{proper-g1}
^{\cal F}{\widehat g}_k^{(n+1)}=& ^{\cal F}{\widehat g}_{k-1}^{(n)}, \quad 2\leq k\leq n,\\\label{proper-g2}
^{\cal F}{\widehat g}_1^{(n+1)}=& ^{\cal F}{\widehat g}_0^{(n)}+b_n,
\end{align}
where $b_n=\dsum_{r=0}^m\lambda_r\dsum_{j=1}^{N^{(\beta_r)}}{\widehat \omega_j^{(\beta_r)}}e^{-(n-1)s_j^{(\beta_r)}\tau}B_j^{(\beta_r)}$.
Then, for $n\geq 1$,

\begin{align*}
{^{\cal FH}}{\cal {\widehat D}}_t^\beta v^{n-1+\sigma}=\sum_{k=0}^{n-1} {^{\cal F}{\widehat g}_k^{(n)}}(v^{k+1}-v^k)
&= \sum_{k=1}^{n} {^{\cal F}{\widehat g}_{k-1}^{(n)}}(v^{k}-v^{k-1})\\
&= \sum_{k=2}^{n} {^{\cal F}{\widehat g}_{k}^{(n+1)}}(v^{k}-v^{k-1})+{^{\cal F}{\widehat g}_0^{(n)}}(v^1-v^0)\\
&= \sum_{k=2}^{n} {^{\cal F}{\widehat g}_{k}^{(n+1)}}(v^{k}-v^{k-1})+{^{\cal F}{\widehat g}_1^{(n+1)}}(v^1-v^0)-b_n (v^1-v^0)\\
&= \sum_{k=1}^{n} {^{\cal F}{\widehat g}_{k}^{(n+1)}}(v^{k}-v^{k-1})-b_n (v^1-v^0).
\end{align*}
For simplicity of representation, we take
\begin{align}\label{vk+1-sigma}
&v^{k+1-\sigma}=(1-\sigma)v^{k+1}+\sigma v^k,~~ k\geq0,\\\label{bfg}
&{\bf g}_0^{(n+1)}={^{\cal F}{\widehat  g}_0^{(n+1)}}-\frac12 b_n,\qquad {\bf g}_k^{(n+1)}={^{\cal F}{\widehat  g}_k^{(n+1)}},\quad k\geq 1.
\end{align}
Invoking the weights investigated in Lemma \ref{weigted-app}, we have
\begin{align*}
&(1-\sigma) {^{\cal FH}}{\cal {\widehat D}}_t^\beta v^{n+\sigma}+\sigma {^{\cal FH}}{\cal {\widehat D}}_t^\beta v^{n-1+\sigma}\\
=& (1-\sigma) \sum_{k=0}^n  {^{\cal F}{\widehat g}_k^{(n+1)}}(v^{k+1}-v^k) + {\sigma}\sum_{k=1}^{n}  {^{\cal F}{\widehat g}_k^{(n+1)}}(v^{k}-v^{k-1})-\sigma b_n(v^1-v^0) \\
=&\sum_{k=1}^n  {\bf g}_k^{(n+1)}(v^{k+1-\sigma}-v^{k-\sigma}) +\left({^{\cal F}{\widehat g}_0^{(n+1)}}-\frac{\sigma}{1-\sigma} b_n\right)(v^{1-\sigma}-v^0).
\end{align*}
However, we find that the last coefficient ${^{\cal F}{\widehat g}_0^{(n+1)}}-\dfrac{\sigma}{1-\sigma} b_n$ is not always positive, which may cause difficulty in our analysis. In view of this, we regroup the last term
\begin{align*}
&\left[{^{\cal F}{\widehat g}_0^{(n+1)}}-\frac{\sigma}{1-\sigma} b_n\right](v^{1-\sigma}-v^0)\\
=&\left[{(1-\sigma)}{^{\cal F}{\widehat g}_0^{(n+1)}}-{\sigma} b_n\right](v^1-v^0)\\
=&\left[(1-\sigma)({^{\cal F}{\widehat g}_0^{(n+1)}}-b_n)-(2\sigma-1)b_n \right](v^1-v^0)\\
=&({^{\cal F}{\widehat g}_0^{(n+1)}}-\frac12 b_n)(v^{1-\sigma}-v^0)-(2\sigma-1)b_n (v^1-v^0)
-\frac12 b_n (v^{1-\sigma}-v^0)\\
=&{\bf g}_0^{(n+1)}(v^{1-\sigma}-{ v^{-\sigma}}),
\end{align*}
where
$${ v^{-\sigma}}\triangleq {\tilde b_n}(v^{1-\sigma}-v^0)+v^0 \quad \mbox{with} \quad {\tilde b_n}=\frac{(3\sigma-1)b_n}{2(1-\sigma){\bf g}_0^{(n+1)}}.$$
Thus
\begin{align}\label{weighted-discretization}
(1-\sigma) {^{\cal FH}}{\cal {\widehat D}}_t^\beta v^{n+\sigma}+\sigma {^{\cal FH}}{\cal {\widehat D}}_t^\beta v^{n-1+\sigma}
=  \sum_{k=0}^n  {\bf g}_k^{(n+1)}(v^{k+1-\sigma}-v^{k-\sigma}), \quad 1\leq n\leq N-1 .
\end{align}
Note that the function $v$ is still not discretized by $u$ in the above derivations, to obtain a fully discretization for  ${_\lambda D}_t^\alpha u(t_{n})~(1\leq n\leq N-1)$, we next first verify some necessary properties of the coefficients.

\begin{lemma}\label{coefficient-pro1}
The coefficients in \eqref{bfg} and $b_n$ satisfy $(n\geq 1)$
\begin{align*}
&(a)~Ct_{n+\sigma}^{1-\alpha_{i_1}}\leq
{\bf g}_0^{(n+1)} \leq Ct_{n-1+\sigma}^{1-\alpha_{i_2}};~~~~\quad (b)~{b}_n < 2{\bf g}_0^{(n+1)};  \\
&(c)~\tau\sum_{k=1}^n{\bf g}_k^{(n+1)}\leq Ct_{n-1}^{2-\alpha_0},~~n\geq2;
~~~~~ (d)~{\bf g}_1^{(n+1)}\leq C{\bf g}_0^{(n+1)};
\end{align*}
where $\alpha_{i_1},\alpha_{i_2}\in(\alpha_m,\alpha_0)$.
\end{lemma}

\begin{proof}

$\bullet$ {\bf Estimation of $(a)$:} Since
$$A_j^{(\beta_r)}-\frac12 B_j^{(\beta_r)}=\int_0^1\left(\frac74-\frac32 s\right)e^{-s_j^{(\beta_r)}\tau(\sigma+1-s)} ds=e^{-s_j^{(\beta_r)}\tau(\sigma+1-s_0)},~~s_0\in(0,1),$$
then
\begin{align}\label{g0n11}
{\bf g}_0^{(n+1)}={^{\cal F} {{\widehat g}_0^{(n+1)}}}-\frac12 b_n=&\sum_{r=0}^m\lambda_r\sum_{j=1}^{N^{(\beta_r)}}{\widehat \omega_j^{(\beta_r)}}e^{-(n-1)s_j^{(\beta_r)}\tau}(A_j^{(\beta_r)}-\frac12 B_j^{(\beta_r)})\\\notag
=&\sum_{r=0}^m\lambda_r\sum_{j=1}^{N^{(\beta_r)}}{\widehat \omega_j^{(\beta_r)}}e^{-(n+\sigma-s_0)s_j^{(\beta_r)}\tau}\\\label{g0n12}
=&\sum_{r=0}^m\lambda_r((n+\sigma-s_0)\tau)^{-\beta_r}+{\cal O}(\epsilon).
\end{align}
Hence, there exist $\alpha_{i_0},\alpha_{i_1}\in(\alpha_m,\alpha_0)$ such that
$$Ct_{n+\sigma}^{1-\alpha_{i_0}}\leq {\bf g}_0^{(n+1)} \leq Ct_{n-1+\sigma}^{1-\alpha_{i_1}}.$$
$\bullet$ {\bf Estimation of $(b)$:} Note that
\begin{align*}
A_j^{(\beta_r)}-B_j^{(\beta_r)}=&\int_0^1\left(2-2s\right)e^{-s_j^{(\beta_r)}\tau(\sigma+1-s)} ds>0,
\end{align*}
that is $B_j^{(\beta_r)}<2\left(A_j^{(\beta_r)}-\frac12 B_j^{(\beta_r)}\right)$, then it follows from \eqref{g0n11} that
$$b_n=\sum_{r=0}^m\lambda_r\sum_{j=1}^{N^{(\beta_r)}}{\widehat \omega_j^{(\beta_r)}}e^{-(n-1)s_j^{(\beta_r)}\tau}B_j^{(\beta_r)}<2{\bf g}_0^{(n+1)}. $$
$\bullet$ {\bf Estimation of $(c)$:}
We first notice that
$$
A_j^{(\beta_r)}\leq \frac32 e^{-s_j^{(\beta_r)}\tau\sigma},~~B_j^{(\beta_r)}\leq \frac12 e^{-s_j^{(\beta_r)}\tau\sigma}.$$
Then, for $1\leq k\leq n-1$ we have
\begin{align*}
{^{\cal F}{g_k^{(n+1,\beta_r)}}} =&\sum_{j=1}^{N^{(\beta_r)}}{\widehat \omega_j^{(\beta_r)}}\left( e^{-(n-k-1)s_j^{(\beta_r)}\tau}A_j^{(\beta_r)}+
 e^{-(n-k)s_j^{(\beta_r)}\tau}B_j^{(\beta_r)}\right) \\
\leq& \sum_{j=1}^{N^{(\beta_r)}}{\widehat \omega_j^{(\beta_r)}}\left( \frac32e^{-(n-k-1+\sigma)s_j^{(\beta_r)}\tau}+
 \frac12e^{-(n-k+\sigma)s_j^{(\beta_r)}\tau}\right) \\
\leq& 2t_{n-k-1+\sigma}^{-\beta_r}+{\cal O}(\epsilon)
\leq   Ct_{n-k-1+\sigma}^{-\beta_r},
 \end{align*}
and for $k=n$
$${^{\cal F}{g_n^{(n+1,\beta_r)}}} =\sum_{j=1}^{N^{(\beta_r)}}{\widehat \omega_j^{(\beta_r)}}B_j^{(\beta_r)}+\lambda a_0^{(\beta_r)}\leq
\frac12 \sum_{j=1}^{N^{(\beta_r)}}{\widehat \omega_j^{(\beta_r)}}e^{-s_j^{(\beta_r)}\tau\sigma}+ C\tau^{-\beta_r}\leq Ct_{\sigma}^{-\beta_r}.$$
Therefore
\begin{align*}
\tau\sum_{k=1}^n{^{\cal F}{g_k^{(n+1,\beta_r)}}} \leq  C\tau\left(\sum_{k=1}^{n-1} t_{n-k-1+\sigma}^{-\beta_r}
+t_{\sigma}^{-\beta_r}\right)
\leq C\tau^{1-\beta_r}\sum_{k=2}^n(n-k+\sigma)^{-\beta_r}\\
\leq C\tau^{1-\beta_r}\int_0^{n-1}x^{-\beta_r}dx
\leq C\tau^{1-\beta_r}(n-1)^{1-\beta_r}
=Ct_{n-1}^{2-\alpha_r}.
\end{align*}
Thus $(c)$ can be verified by the fact
\begin{align}\label{coe-relation}
{\bf g}_k^{(n+1)}={^{\cal F}{{\widehat g}_k^{n+1}}}=\sum_{r=0}^m\lambda_r {^{\cal F}{g_k^{(n+1,\beta_r)}}}, \quad 1\leq k\leq n.
\end{align}
$\bullet$ {\bf Estimation of $(d)$:}
As $A_j^{(\beta_r)}>0$, for  $n\geq 2$ we have
\begin{align}\notag
^{\cal F}{\widehat g}_0^{(n)}=\sum_{r=0}^m\lambda_r \sum_{j=1}^{N^{(\beta_r)}}{\widehat \omega_j^{(\beta_r)}} e^{-(n-2)s_j^{(\beta_r)}\tau}A_j^{(\beta_r)}
\leq&\max_{j,r}\{e^{s_j^{(\beta_r)}\tau}\} \sum_{r=0}^m\lambda_r \sum_{j=1}^{N^{(\beta_r)}}{\widehat \omega_j^{(\beta_r)}} e^{-(n-1)s_j^{(\beta_r)}\tau}A_j^{(\beta_r)}\\\label{eqq}
\leq& C~ {^{\cal F}}{\widehat g}_0^{(n+1)},
\end{align}
for $n=1$, noticing \eqref{g0n12}, we have
\begin{align}\label{eqq01}
^{\cal F}{\widehat g}_0^{(1)}=\sum_{r=0}^m\lambda_r \mu^{(\beta_r)}a_0^{(\beta_r)}
=\sigma \sum_{r=0}^m \frac{\lambda_r}{\Gamma(2-\beta_r)}(\sigma\tau)^{-\beta_r}
\leq C~{^{\cal F}}{\widehat g}_0^{(2)}.
\end{align}
Thus by $(b)$ and \eqref{eqq}--\eqref{eqq01},
we get the verification of $(d)$:
\begin{align*}
{\bf g}_1^{(n+1)}=&^{\cal F}{\widehat g}_0^{(n)}+b_n
\leq C~{^{\cal F}{{\widehat g}_0^{(n+1)}}}+b_n
\leq C {\bf g}_0^{(n+1)}.
\end{align*}
\end{proof}
Taking
\begin{align}\label{hatvu1}
&{\widehat v}^{k+1-\sigma}=(2-2\sigma)\delta_t u^{k+\frac12}+(2\sigma-1)\delta_{\widehat t}u^{k},\quad 0\leq k\leq n,\\\label{hatvu2}
&{\widehat v}^{1-\sigma}=(2-2\sigma)\delta_tu^{\frac12}+(2\sigma-1)u_t^0,\\\label{hatvu3}
&{\widehat v}^{-\sigma}={\tilde b}_n({\widehat v}^{1-\sigma}-u_t^0)+u_t^0.
\end{align}
From the subsection 2.2 in \cite{LyuVongNA2017}, we know that
\begin{align}\nonumber
&v^{k+1-\sigma}={\widehat v}^{k+1-\sigma}-(2-2\sigma)R_t^{k+\frac12}-(2\sigma-1)R_{\widehat t}^{k}+(2-2\sigma)R_{v}^{k+\frac12},\\\label{vktheta}
&v^{1-\sigma}={\widehat v}^{1-\sigma}-(2-2\sigma)R_t^{\frac12}+(2-2\sigma)R_{v}^{\frac12},
\end{align}
where the truncation errors satisfy
$$R_{v}^{k+\frac12}={\cal O}(\tau^2),\quad R_t^{k+\frac12}={\cal O}(\tau^2), \quad
R_{\widehat t}^{k+1}={\cal O}(\tau^2)\quad \mbox{for}\quad 0\leq k \leq n,$$
provided $u(t)\in {\cal C}^3[0,T]$. Moreover,
$$R_{v}^{k+\frac12}-R_{v}^{k-\frac12}={\cal O}(\tau^3),~ R_t^{k+\frac12}-R_t^{k-\frac12}=
{\cal O}(\tau^3),~ R_{\widehat t}^{k+1}-R_{\widehat t}^{k}={\cal O}(\tau^3) \quad \mbox{for}\quad 1\leq k\leq n,$$
provided $u(t)\in {\cal C}^4[0,T]$.

We are now ready to show a fully discretization to ${_\lambda D}_t^\alpha u(t_{n})~(1\leq n\leq N-1)$ and its truncation error.
\begin{lemma}\label{main-lemma}
Suppose $u(t)\in{\cal C}^4[0,T]$. Denote
\begin{align}\label{full-appro}
{^{\cal FH}}{\cal {\widehat D}}_t^\alpha u^{n+1}=\sum_{k=0}^n  {\bf g}_k^{(n+1)}({\widehat v}^{k+1-\sigma}-{\widehat v}^{k-\sigma}), \quad 1\leq n\leq N-1 .
\end{align}
Then
$$\left|{_\lambda D}_t^\alpha u(t_{n})-{^{\cal FH}}{\cal {\widehat D}}_t^\alpha u^{n+1} \right|\leq
C\left({\bf g}_0^{(n+1)}\tau^2+\epsilon \right), \quad 1\leq n\leq N-1 .$$
\end{lemma}

\begin{proof}
By \eqref{weighted-discretization}  and \eqref{vktheta}--\eqref{full-appro}, we can derive

\begin{align}\notag
(1-\sigma) ~{^{\cal FH}}{\cal {\widehat D}}_t^\beta v^{n+\sigma}+\sigma~ {^{\cal FH}}{\cal {\widehat D}}_t^\beta v^{n-1+\sigma}={^{\cal FH}}{\cal {\widehat D}}_t^\alpha u^{n+1}-{\tilde R}_t^{n+1}, \quad 1\leq n\leq N-1,
\end{align}
where
\begin{align*}
{\tilde R}_t^{n+1} =&\sum_{k=2}^n  {\bf g}_k^{(n+1)}\bigg[(2-2\sigma)\left(R_t^{k+\frac12}-R_t^{k-\frac12}\right)+(2\sigma-1)\left(R_{\widehat t}^{k}-R_{\widehat t}^{k-1} \right)
-(2-2\sigma)\left(R_{v}^{k+\frac12}-R_{v}^{k-\frac12} \right)\bigg]\\
&+{\bf g}_1^{(n+1)}\bigg[(2-2\sigma)\left(R_t^{\frac32}-R_t^{\frac12}\right)-(2-2\sigma)\left(R_{v}^{\frac32}-R_{v}^{\frac12} \right)+ (2\sigma-1)R_{\widehat t}^1\bigg]\\
&+{\bf g}_0^{(n+1)}(1-{\tilde b^n})(2-2\sigma)\left(R_t^{\frac12}-R_{v}^{\frac12}\right).
\end{align*}

With the help of Lemma \ref{coefficient-pro1} (b), (c) and (d), we have
\begin{align*}
\big|{\tilde R}_t^{n+1}\big|
 \leq & (2-2\sigma)\sum_{k=1}^n {\bf g}_k^{(n+1)}\Big|R_t^{k+\frac12}-R_t^{k-\frac12}\Big|+(2\sigma-1)\sum_{k=2}^n {\bf g}_k^{(n+1)}\Big|R_{\widehat t}^{k}-R_{\widehat t}^{k-1}\Big|\\
 & +(2-2\sigma)\sum_{k=1}^n {\bf g}_k^{(n+1)} \Big|R_{v}^{k+\frac12}-R_{v}^{k-\frac12}\Big|+ (2\sigma-1){\bf g}_1^{(n+1)}\big|R_{\widehat t}^1\big|\\
&+(2-2\sigma){\bf g}_0^{(n+1)}|1-{\tilde b^n}|\left( \big|R_t^{\frac12}\big|+\big|R_{v}^{\frac12}\big|\right)\\
 \leq & C\sum_{k=1}^n {\bf g}_k^{(n+1)}\tau^3+C({\bf g}_0^{(n+1)}+{\bf g}_1^{(n+1)}) \tau^2
 \leq C{\bf g}_0^{(n+1)}\tau^2,
\end{align*}
in which  ${\bf g}_0^{(n+1)}{\tilde b^n}=\frac{3\sigma-1}{2(1-\sigma)}b_n\leq C {\bf g}_0^{(n+1)}$ has been used.

Applying the relation \eqref{relation-Dv-Du}, Lemmas \ref{appro-Dv} and \ref{weigted-app}, we have
\begin{align*}
{_\lambda D}_t^\alpha u(t_{n})={_\lambda D}_t^\beta v(t_{n})
=&(1-\sigma){_\lambda D}_t^\beta v(t_{n+\sigma})+\sigma {_\lambda D}_t^\beta v(t_{n-1+\sigma})
+{\cal O}(\tau^2)\\
=&(1-\sigma) {^{\cal FH}}{\cal {\widehat D}}_t^\beta v^{n+\sigma}+\sigma {^{\cal FH}}{\cal {\widehat D}}_t^\beta v^{n-1+\sigma}+{\cal O}(\tau^2+\epsilon).
\end{align*}
Therefore,
\begin{align*}
\left|{_\lambda D}_t^\alpha u(t_{n})-{^{\cal FH}}{\cal {\widehat D}}_t^\alpha u^{n+1} \right|=&
\left| (1-\sigma) {^{\cal FH}}{\cal {\widehat D}}_t^\beta v^{n+\sigma}+\sigma {^{\cal FH}}{\cal {\widehat D}}_t^\beta v^{n-1+\sigma}-{^{\cal FH}}{\cal {\widehat D}}_t^\alpha u^{n} \right|
+{\cal O}(\tau^2+\epsilon)\\
=&\big|{\tilde R}_t^{n+1}\big|+{\cal O}(\tau^2+\epsilon)
\leq C\left({\bf g}_0^{(n+1)}\tau^2 +\epsilon \right).
\end{align*}
\end{proof}

\subsubsection{The first time level approximation}
Note that the discretization \eqref{full-appro} will be used to solve the grid function $u^{n+1}~(n\geq 1)$. For the first level grid function, denote ${\bf g}_0^{(1)}=\frac{{^{\cal F} {{\widehat g}_0^{(1)}}}}{1-\sigma}$ and
$${^{\cal FH}}{\cal {\widehat D}}_t^\alpha u^{1}={\bf g}_0^{(1)}({\widehat v}^{1-\sigma}-u_t^0).$$
Then it can be observed that
$${^{\cal FH}}{\cal {\widehat D}}_t^\beta v^{\sigma}={^{\cal FH}}{\cal {\widehat D}}_t^\alpha u^{1}+(2-2\sigma){\bf g}_0^{(1)}(R_v^{\frac12}-R_t^{\frac12}).$$
By the relation \eqref{relation-Dv-Du} and Lemma \ref{appro-Dv}
, we further get
\begin{align}\label{first-level-appro}
{_\lambda D}_t^\alpha u(t_{\sigma})={_\lambda D}_t^\beta v(t_{\sigma})={^{\cal FH}}{\cal {\widehat D}}_t^\beta v^{\sigma}+R^\sigma={^{\cal FH}}{\cal {\widehat D}}_t^\alpha u^{1}+R^\sigma+R^1,
\end{align}
where
\begin{align}\label{Rsigma}
R^\sigma={\cal O}(\tau^{3-\beta_0})\quad  \mbox{and}\quad
R^1=(2-2\sigma){\bf g}_0^{(1)}(R_v^{\frac12}-R_t^{\frac12})\leq C{\bf g}_0^{(1)}\tau^2.
\end{align}

\subsection{The fast linearized scheme}\label{scheme-section}
Before proposing our numerical scheme, we need an important lemma which provides a fitted approximation (as it plays a significant role in our later convergence analysis) to the time discretization on diffusion term. The lemma reads as:
\begin{lemma} \label{implictscheme} {\rm(\cite{LyuVongNA2017})}
Suppose $u(t)\in{\cal C}^2[0,T]$. For $n\ge1$, it holds that
$$u(t_n)=\dfrac{w^{n+1}+w^n}2+{\cal O}(\tau^2),$$
where
\begin{align*}
w^n=&\Big(\frac32-\sigma\Big)[\sigma u^n+(1-\sigma)u^{n-1}]
+\Big(\sigma-\frac12\Big)[\sigma u^{n-1}+(1-\sigma)u^{n-2}],~~n\ge2,\\
w^1=&\Big(\frac32-\sigma\Big)[\sigma u^1+(1-\sigma)u^{0}]
+\Big(\sigma-\frac12\Big)[\sigma u^{0}+(1-\sigma)(u^{1}-2\tau u_t^0)].
\end{align*}
\end{lemma}

For a given positive integer $M$, denote $h=\frac{b-a}M$ be the spatial step size, and ${\bar I}=\{i| ~i=0,1,\ldots,M\}$ and $I=\{i| ~i=1,2,\ldots,M-1\}$ be the index spaces, and take $x_i=ih~(i\in {\bar I})$ be the spatial uniform partition.
Denote $\varphi_i=\varphi(x_i)$ and $\psi_i=\psi(x_i)$. Suppose ${\cal V}_h=\big\{u|u=\{u_i|0\leq i\leq M\},~u_0=u_M=0\big\}$, then for any $u,~v\in{\cal V}_h$, the discrete operators are needed:
$$\delta_xu_{i-\frac12}=\frac{u_i-u_{i-1}}h,~~\delta_x^2u_i=\frac{\delta_xu_{i+\frac12}-\delta_xu_{i-\frac12}}h=\frac{u_{i+1}-2u_{i}+u_{i-1}}h.$$
Denote the numerical solution at the point $(x_i,t_n)$ by $u_i^n$, and $p(x_i,t_n)$ by $p_i^n$.

Considering the equation \eqref{eq1}--\eqref{eq3} on the grid point $(x_i,t_n)$:
\begin{align*}
{_\lambda D}_t^\alpha u(x_i,t_{n})=\frac{\partial^2 u}{\partial x^2}(x_i,t_n)+f(u(x_i,t_n))+p_i^n,\quad i\in I, ~1\leq n\leq N-1.
\end{align*}
Applying Lemmas \ref{main-lemma} and \ref{implictscheme}, and standard approximation on space derivative, the above equation yields
\begin{align}\label{equation-point-n}
{^{\cal FH}}{\cal {\widehat D}}_t^\alpha u_i^{n+1}=\delta_x^2\left(\frac{w_i^{n+1}+w_i^n}{2}\right)+f(u_i^n)+p_i^n+R_i^{n+1},\quad i\in I, ~1\leq n\leq N-1,
\end{align}
where
$$|R_i^{n+1}|\leq C\left({\bf g}_0^{(n+1)}\tau^2+h^2+\epsilon \right).$$
The equation \eqref{equation-point-n} solves the solution $u_i^{n+1}$, thus it is a linearized equation since the nonlinear term is stayed on the previous time level $t_n$.

For the first level solution, considering the equation \eqref{eq1}--\eqref{eq3} on the grid point $(x_i,t_\sigma)$:
\begin{align}\label{equation-point-1p}
{_\lambda D}_t^\alpha u(x_i,t_{\sigma})=\frac{\partial^2 u}{\partial x^2}(x_i,t_\sigma)+f(u(x_i,t_\sigma))+p_i^\sigma,\quad i\in I.
\end{align}
With the help of Taylor expansion, we use the initial data to approximate $u(x_i,t_\sigma)$:
\begin{align}\label{sigma-appro}
u(x_i,t_\sigma)=u(x_i,t_0)+\sigma\tau u_t(x_i,t_0)+{\cal O}(\tau^2)=\varphi_i+\sigma\tau \psi_i +{\cal O}(\tau^2), \quad i\in I.
\end{align}
Then for $\Omega_f\supset [\inf\{\varphi+\sigma\tau\psi\},\sup\{\varphi+\sigma\tau\psi\}]$,
it follows from \eqref{equation-point-1p},  \eqref{first-level-appro} and \eqref{sigma-appro} that
\begin{align}\label{equation-point-1}
{^{\cal FH}}{\cal {\widehat D}}_t^\alpha u_i^{1}=(\varphi_{xx}+\sigma\tau\psi_{xx})_i
+f\left(\varphi_i+\sigma\tau\psi_i\right)+p_i^\sigma+{\tilde R}_i^\sigma+R_i^1,\quad i\in I,
\end{align}
where ${\tilde R}_i^\sigma={\cal O}(\tau^{3-\beta_0}+\tau^2)={\cal O}(\tau^2)$. We can see that equation \eqref{equation-point-1} is also a linearized approximation.

Omitting $R_i^{n+1}$ in \eqref{equation-point-n}, and ${\tilde R}_i^\sigma,R_i^1$ in \eqref{equation-point-1}, we obtain the following fully fast linearized scheme for solving the nonlinear problem \eqref{eq1}--\eqref{eq3}:
\begin{align}\label{sch1}
&{^{\cal FH}}{\cal {\widehat D}}_t^\alpha u_i^{n+1}=\delta_x^2\left(\frac{w_i^{n+1}+w_i^n}{2}\right)+f(u_i^n)+p_i^n,\quad i\in I, ~1\leq n\leq N-1,\\\label{sch2}
&{^{\cal FH}}{\cal {\widehat D}}_t^\alpha u_i^{1}=(\varphi_{xx}+\sigma\tau\psi_{xx})_i
+f\left(\varphi_i+\sigma\tau\psi_i\right)+p_i^\sigma,\quad i\in I,\\\label{sch3}
&u_0^n=u_M^n=0,~0\leq n\leq N,\\\label{sch4}
&u_i^0=\varphi_i,~(u_t)_i^0=\psi_i,~i\in {\bar I}.
\end{align}
\begin{remark}
In order to interpret the fast algorithm of the proposed scheme \eqref{sch1}--\eqref{sch4}, we display the efficient computation for the fractional discretization ${^{\cal FH}}{\cal {\widehat D}}_t^\alpha u_i^{n+1}$ based on the recursive relation \eqref{Fast-dv} and the refined grid functions \eqref{hatvu1}--\eqref{hatvu3}:

For $n=0,$ we can obtain $u_i^1$, $v_i^1$ {\rm(}combining \eqref{for-v1}{\rm)} by solving \eqref{sch2}.

For $n=1,$ we can solve the following system to obtain $u_i^2,~v_i^2$:
\begin{equation}\notag
\left\{\begin{array}{ll}
\dsum_{r=0}^m\lambda_r\dsum_{j=1}^{N^{(\beta_r)}}{\widehat\omega}_j^{(\beta_r)}(1-\sigma)V_j^{1}+{\bf a_0}(v_i^{2-\sigma}-v_i^{1-\sigma})
=\delta_x^2\left(\frac{w_i^{2}+w_i^1}{2}\right)+f(u_i^1)+p_i^1,\\
(1-\sigma)v_i^2+\sigma v_i^1=(2-2\sigma)\delta_t u_i^{\frac32}+(2\sigma-1)\delta_{\widehat t}u_i^{1},~~i\in I,
\end{array}\right.
\end{equation}
where ${\bf a_0}=\dsum_{r=0}^m\lambda_r \mu^{(\beta_r)}a_0^{(\beta_r)},~V_j^{1}=A_j^{(\beta_r)}(v_i^{1}-v_i^{0})+B_j^{(\beta_r)}(v_i^{2}-v_i^1),$ and $v_i^{2-\sigma},~v_i^{1-\sigma}$ are defined by \eqref{vk+1-sigma}.

For $2\leq n\leq N-1$,
\begin{align*}
&{^{\cal FH}}{\cal {\widehat D}}_t^\alpha u_i^{n+1}=\sum_{r=0}^m\lambda_r\dsum_{j=1}^{N^{(\beta_r)}}{\widehat\omega}_j^{(\beta_r)} V_j^{n-\sigma}+{\bf a_0}({\widehat v_i^{n+1-\sigma}}-{\widehat v_i^{n-\sigma}}),~~i\in I,
\end{align*}
where a recursive relation
\begin{align*}
V_j^{n-\sigma}=
e^{-s_i^{(\beta_r)}\tau}V_j^{n-1-\sigma} +A_j^{(\beta_r)}({\widehat v_i^{n-\sigma}}-{\widehat v_i^{n-1-\sigma}})+B_j^{(\beta_r)}({\widehat v_i^{n+1-\sigma}}-{\widehat v_i^{n-\sigma}}),
\end{align*}
with $V_j^{1-\sigma}=(1-\sigma)\Big[A_j^{(\beta_r)}(v_i^{1}-v_i^{0})+B_j^{(\beta_r)}(v_i^{2}-v_i^1)\Big]$.
\end{remark}
\section{Analysis of the proposed scheme}\label{Sec-conver}

\subsection{Some necessary lemmas}

To show the convergence of proposed scheme, we need some important lemmas.

\begin{lemma} \label{lemmafast} {\rm(\cite{FL21})}
For the sequence $\{{^{\cal F}{g_k^{(n+1,\beta)}}}\}~(k=0,\ldots,n)$, it holds that
\begin{equation}\notag
(2\sigma-1)~{^{\cal F}{g_n^{(n+1,\beta)}}}-\sigma~{^{\cal F}{g_{n-1}^{(n+1,\beta)}}}\geq\left\{\begin{array}{ll}
(2\sigma-1){{g}_n^{(n+1,\beta)}}-\sigma{{g}_{n-1}^{(n+1,\beta)}}-(6\sigma-1)\frac{\epsilon}{4\Gamma(1-\beta)},~n=1,\\
(2\sigma-1){{g}_n^{(n+1,\beta)}}-\sigma{{g}_{n-1}^{(n+1,\beta)}}-(7\sigma-1)\frac{\epsilon}{4\Gamma(1-\beta)},~n\geq2,
\end{array}\right.
\end{equation}
and
$${^{\cal F} { g_n^{(n+1,\beta)}}}>{^{\cal F} { g_{n-1}^{(n+1,\beta)}}}>\cdots>{^{\cal F} {g_0^{(n+1,\beta)}}}>C^{\cal F}>0,$$
where $C^{\cal F}=\min\{\mu^{(\beta)} a_0^{(\beta)},{^{\cal F}}g_0^{(n+1,\beta)}\}.$
\end{lemma}

\begin{lemma}\label{lemmamultiterm}{\rm(\cite{multiterm})}
There exists a number $\tau_0>0$, when $\tau \leq\tau_0$, it holds that
\begin{align*}
(2\sigma-1){\widehat g}_n^{(n+1)}-\sigma{\widehat g}_{n-1}^{(n+1)}>\frac12\lambda_0\frac{\tau^{-\beta_0}}{\Gamma(2-\beta_0)}(1+\sigma)^{-\beta_0}[f_\sigma(1)+{\cal O}(\tau^{\beta_0-\beta_1})]>0,\quad
n\geq1,
\end{align*}
where 
$$f_\sigma(t)=\Big(3\sigma^2+5\sigma+2-\frac1{\sigma}\Big)t+\frac{2\sigma^2}t-5\sigma^2-7\sigma+1.$$
\end{lemma}

\begin{lemma}\label{multi-coeff-property}
There exists a number $\tau_0>0$, when $\tau \leq\tau_0$, it holds that {\rm($n\geq1$)}
\begin{align}
&{\bf g}_n^{(n+1)}>{\bf g}_{n-1}^{(n+1)}>\ldots>{\bf g}_0^{(n+1)}>0,\label{ineq1}\\
&(2\sigma-1){\bf g}_n^{(n+1)}-\sigma {\bf g}_{n-1}^{(n+1)}>0,\label{ineq2}
\end{align}
for a sufficiently small $\epsilon$.
\end{lemma}
\begin{proof}
From the relation \eqref{coe-relation} and Lemma \ref{lemmafast}, we easily get
$${\bf g}_n^{(n+1)}>{\bf g}_{n-1}^{(n+1)}>\cdots>{\bf g}_1^{(n+1)}>0.$$
Notice that
$${\bf g}_0^{(n+1)}={^{\cal F} {{\widehat g}_0^{(n+1)}}}-\frac12 b_n<{^{\cal F} {{\widehat g}_0^{(n+1)}}}
<{^{\cal F} {{\widehat g}_1^{(n+1)}}}={\bf g}_1^{(n+1)},$$
with property $(a)$ in Lemma \ref{coefficient-pro1},
\eqref{ineq1} is obtained.

For $n\geq2$,
then by Lemma \ref{lemmafast}, we have
\begin{align}\nonumber
(2\sigma-1){\bf g}_n^{(n+1)}-\sigma {\bf g}_{n-1}^{(n+1)}
=&(2\sigma-1)~{^{\cal F} {\widehat g_{n}^{(n+1)}}}-\sigma~{^{\cal F} {\widehat g_{n-1}^{(n+1)}}}\\\nonumber
\label{coeff}=&\dsum_{r=0}^m\lambda_r\left[(2\sigma-1)~{^{\cal F} { g_{n}^{(n+1,{\beta_r})}}}-\sigma~{^{\cal F} { g_{n-1}^{(n+1,{\beta_r})}}}\right]\\
\geq&\sum_{r=0}^m\lambda_r\left[(2\sigma-1) g_n^{(n+1,\beta_r)}-\sigma g_{n-1}^{(n+1,\beta_r)}\right]-\frac{7\sigma-1}4\epsilon\dsum_{r=0}^m\frac{\lambda_r}{\Gamma(1-\beta_r)}.
\end{align}
By Lemma \ref{lemmamultiterm}, we reach
\begin{align*}
\sum_{r=0}^m\lambda_r\left[(2\sigma-1)g_n^{(n+1,\beta_r)}-\sigma g_{n-1}^{(n+1,\beta_r)}\right]=&
(2\sigma-1){\widehat g}_n^{(n+1)}-\sigma {\widehat g}_{n-1}^{(n+1)}\\
> &\frac{\lambda_0\tau^{-\beta_0}}{2\Gamma(2-\beta_0)}(1+\sigma)^{-\beta_0}[f_\sigma(1)+{\cal O}(\tau^{\beta_0-\beta_1})]>0.
\end{align*}
Combining the above inequality with \eqref{coeff}, we can conclude that
$$(2\sigma-1){\bf g}_n^{(n+1)}-\sigma {\bf g}_{n-1}^{(n+1)}>0,\quad n\geq 2,$$
holds as long as
$$\epsilon<\left\{\frac12\lambda_0\frac{\tau^{-\beta_0}}{\Gamma(2-\beta_0)}(1+\sigma)^{-\beta_0}[f_\sigma(1)+{\cal O}(\tau^{\beta_0-\beta_1})]\right\}\Big/ \left[\frac{7\sigma-1}4\dsum_{r=0}^m\frac{\lambda_r}{\Gamma(1-\beta_r)}\right].$$
For $n=1$, using the similar way as $n\geq2$, we have
\begin{align}\label{coe-proof2}
(2\sigma-1){\bf g}_1^{(2)}-\sigma {\bf g}_{0}^{(2)}
\geq \frac{\lambda_0\tau^{-\beta_0}(1+\sigma)^{-\beta_0}}{2\Gamma(2-\beta_0)}[f_\sigma(1)+{\cal O}(\tau^{\beta_0-\beta_1})]-\frac{(6\sigma-1)\epsilon}4\dsum_{r=0}^m\frac{\lambda_r}{\Gamma(1-\beta_r)}+\sigma b_1.
\end{align}
Next we show $b_1>0$. Note that
$$b_1
=\sum_{r=0}^m\lambda_r\sum_{j=0}^{N^{(\beta_r)}}{\widehat w_j^{(\beta_r)}}e^{-s_j^{(\beta_r)}\tau(\sigma+1)}\int_0^1(s-\frac12)e^{s_j^{(\beta_r)}\tau s}ds.$$
For the integral term
\begin{align*}
\int_0^1(s-\frac12)e^{s_j^{(\beta_r)}\tau s}&=\int_0^{\frac12}(s-\frac12)e^{s_j^{(\beta_r)}\tau s}ds+\int_\frac12^1(s-\frac12)e^{s_j^{(\beta_r)}\tau s}ds\\
&=\int_0^{\frac12}(s-\frac12)e^{s_j^{(\beta_r)}\tau s}ds+\int_0^\frac12se^{s_j^{(\beta_r)}\tau(s+\frac12)}ds>\int_0^{\frac12}(2s-\frac12)e^{s_j^{(\beta_r)}\tau s}ds.
\end{align*}
Recursively, the above integral has the lower bound
\begin{align*}
\int_0^1(s-\frac12)e^{s_j^{(\beta_r)}\tau s}>2^{n-1}\int_0^{\frac1{2^{n-1}}}(s-\frac1{2^n})e^{s_j^{(\beta_r)}\tau s}ds.
\end{align*}
It is easy to check that
$$2^{n-1}\int_0^{\frac1{2^{n-1}}}(s-\frac1{2^n})e^{s_j^{(\beta_r)}\tau s}ds>0,\quad \mbox{as}\quad  n\rightarrow\infty,$$
hence $b_1>0$.

Therefore, from \eqref{coe-proof2},  it holds that $(2\sigma-1){\bf g}_1^{(2)}-\sigma {\bf g}_{0}^{(2)}>0$
when
$$\epsilon<\left\{\frac12\lambda_0\frac{\tau^{-\beta_0}}{\Gamma(2-\beta_0)}(1+\sigma)^{-\beta_0}[f_\sigma(1)+{\cal O}(\tau^{\beta_0-\beta_1})]\\
+\sigma b_1\right\}\Big/\left[\frac{6\sigma-1}4\dsum_{r=0}^m\frac{\lambda_r}{\Gamma(1-\beta_r)}\right].$$
Thus in general, \eqref{ineq2} holds if
\begin{align*}
\epsilon<\left\{\frac12\lambda_0\frac{\tau^{-\beta_0}}{\Gamma(2-\beta_0)}(1+\sigma)^{-\beta_0}[f_\sigma(1)+{\cal O}(\tau^{\beta_0-\beta_1})]\right\}\Big/ \left[\frac{7\sigma-1}4\dsum_{r=0}^m\frac{\lambda_r}{\Gamma(1-\beta_r)}\right].
\end{align*}
\end{proof}
In the following parts, we always suppose conditions in Lemma \ref{multi-coeff-property} are fulfilled when required. 
\begin{lemma}\label{alik}{\rm(\cite{AAA})}
If the positive sequence $\{d_k^{(n+1)}|~0\leq k\leq n,n\geq 1\}$ is strictly decrease for $k$, and satisfies $(2\sigma-1)d_n^{(n+1)}-\sigma d_{n-1}^{(n+1)}>0$ for a constant $\sigma\in(0,1)$ . Then
$$
2[\sigma y^{n+1}+(1-\sigma)y^n]\sum_{k=0}^{n}d_{k}^{(n+1)}(y^{k+1}-y^k) \geq \sum_{k=0}^{n}d_{k}^{(n+1)}\big[(y^{k+1})^2-(y^k)^2 \big],\quad n\geq 1.
$$
\end{lemma}

We now present a particular form of Lemma \ref{alik}, which will be used in our analysis.
\begin{lemma}\label{stproof1}
 For any real sequence $F^n$, the following inequality holds:
\begin{align*}
&2[\sigma {\widehat v}^{n+1-\sigma}+(1-\sigma){\widehat v}^{n-\sigma}][{^{\cal FH}}{\cal {\widehat D}}_t^\alpha u^{n+1}-F^n]\\
\geq &\sum_{k=0}^n {\bf g}_{k}^{(n+1)} {({\widehat v}^{k+1-\sigma})}^2-\sum_{k=0}^{n-1} {\bf g}_{k}^{(n)}{({\widehat v}^{k+1-\sigma})}^2-\left({\bf g}_1^{(n+1)}-{\bf g}_0^{(n)} \right){({\widehat v}^{1-\sigma})}^2
     -{\bf g}_0^{(n+1)}\Big({\widehat v}^{-\sigma}+\frac1{{\bf g}_0^{(n+1)}}F^n \Big)^2.
\end{align*}
\end{lemma}

\begin{proof}
By Lemma \ref{multi-coeff-property}, incorporating $d_k^{(n+1)}={\bf g}_k^{(n+1)}$, $y^n={\widehat v}^{n-\sigma}$ $(n\geq 1)$ and $y^0={\widehat v}^{-\sigma}+\frac1{{\bf g}_0^{(n+1)}}F^n$ with Lemma \ref{alik}, and noticing the properties \eqref{proper-g1}--\eqref{proper-g2}, we can get
\begin{align*}
&2[\sigma {\widehat v}^{n+1-\sigma}+(1-\sigma){\widehat v}^{n-\sigma}]({^{\cal FH}}{\cal {\widehat D}}_t^\alpha u^{n+1}-F^n)\\
\geq& \sum_{k=1}^n {\bf g}_{k}^{(n+1)} \big[ {({\widehat v}^{k+1-\sigma})}^2-{({\widehat v}^{k-\sigma})}^2\big]+{\bf g}_0^{(n+1)} \Big[{({\widehat v}^{1-\sigma})}^2-\Big({\widehat v}^{-\sigma}+\frac1{{\bf g}_0^{(n+1)}}F^n \Big)^2\Big] \\
=& \sum_{k=0}^n {\bf g}_{k}^{(n+1)} {({\widehat v}^{k+1-\sigma})}^2-\sum_{k=1}^n {\bf g}_{k}^{(n+1)}{({\widehat v}^{k-\sigma})}^2 -{\bf g}_0^{(n+1)}\Big({\widehat v}^{-\sigma}+\frac1{{\bf g}_0^{(n+1)}}F^n \Big)^2\\
=& \sum_{k=0}^n {\bf g}_{k}^{(n+1)} {({\widehat v}^{k+1-\sigma})}^2-\sum_{k=1}^n {\bf g}_{k-1}^{(n)}{({\widehat v}^{k-\sigma})}^2-\left({\bf g}_1^{(n+1)}-  {\bf g}_0^{(n)}  \right){({\widehat v}^{1-\sigma})}^2
     -{\bf g}_0^{(n+1)}\Big({\widehat v}^{-\sigma}+\frac1{{\bf g}_0^{(n+1)}}F^n \Big)^2\\
=& \sum_{k=0}^n {\bf g}_{k}^{(n+1)} {({\widehat v}^{k+1-\sigma})}^2-\sum_{k=0}^{n-1} {\bf g}_{k}^{(n)}{({\widehat v}^{k+1-\sigma})}^2-\left({\bf g}_1^{(n+1)}-{\bf g}_0^{(n)} \right){({\widehat v}^{1-\sigma})}^2
     -{\bf g}_0^{(n+1)}\Big({\widehat v}^{-\sigma}+\frac1{{\bf g}_0^{(n+1)}}F^n \Big)^2.
\end{align*}
\end{proof}
\begin{lemma}\label{GRW} {\rm{(\cite{Gronwall})}}{\rm(}Gronwall's inequality{\rm)}
Let $\{G_{{j}}\}$ and $\{k_{{j}}\}$ be nonnegative sequences satisfying
$$G_0\leq K,\qquad G_{{j}} \leq K+\sum_{l=0}^{{j}-1}k_l G_l,\quad {j}\geq 1, $$
where $K\geq 0$. Then
$$G_{{j}}\leq K\exp\Big(\sum_{l=0}^{{j}-1}k_l \Big),\quad {j}\geq 1. $$
\end{lemma}

Another important lemma concerns some properties about the coefficients for analysis:
\begin{lemma}\label{coeff-property2}
The coefficients ${\bf g}_0^{(k+1)}~(0\leq k\leq n)$ and ${\bf g}_k^{(n+1)}$  satisfy
\begin{align*}
&(a)\quad C\tau^{1-\alpha_m}\leq {\bf g}_0^{(1)}\leq C\tau^{1-\alpha_0}; ~~\quad\quad
   (b) \quad \tau \sum_{k=0}^{n}{\bf g}_0^{(k+1)}<C\max_{1\leq \gamma\leq2}t_{n}^{2-\gamma};\\
 &(c) \quad \tau\sum_{k=0}^n \frac1{{\bf g}_0^{(k+1)}}< C\max_{1\leq \gamma\leq2}t_{n+\sigma}^{\gamma} ;\quad\quad
    (d) \quad \tau\sum_{k=0}^n\frac{1}{{\bf g}_k^{(n+1)}}\leq Ct_{n+\sigma}^{2-\alpha_{i_1}}.
\end{align*}
\end{lemma}
\begin{proof}
Since  ${\bf g}_0^{(1)}=\frac{^{\cal F}{\widehat g}_0^{(1)}}{1-\sigma}$ and
$${^{\cal F}{\widehat g}_0^{(1)}}=
\sum_{r=0}^m\lambda_r \mu^{(\beta_r)}a_0^{(\beta_r)}= {\cal O}(\tau^{-\beta_*}),\quad
\beta_*\in (\beta_m,\beta_0),$$
then $(a)$ can be easily obtained.

By Lemma \ref{coefficient-pro1} $(a)$, we have
\begin{align*}
\tau\sum_{k=1}^n {\bf g}_0^{(k+1)}  \leq  C\tau\sum_{k=1}^{n} t_{k-1+\sigma}^{1-\alpha_{i_2}}
\leq C\tau^{2-\alpha_{i_2}}\sum_{k=0}^{n-1}(k+\sigma)^{1-\alpha_{i_2}}
\leq C\tau^{2-\alpha_{i_2}}\int_0^{n}x^{1-\alpha_{i_2}}dx
\leq Ct_{n}^{2-\alpha_{i_2}},
\end{align*}
uniting $\tau{\bf g}_0^{(1)} \leq C\tau^{2-\alpha_0}$, $(b)$ is verified .

Applying Lemma \ref{coefficient-pro1} $(a)$,
\begin{align*}
\tau\sum_{k=1}^n\frac1{{\bf g}_0^{(k+1)}} \leq  C\tau\sum_{k=1}^{n} t_{k+\sigma}^{\alpha_{i_1}-1}
\leq C\tau^{\alpha_{i_1}}\sum_{k=1}^{n}(k+\sigma)^{\alpha_{i_1}-1}
\leq C\tau^{\alpha_{i_1}}\int_0^{n+\sigma}x^{\alpha_{i_1}-1}dx
\leq Ct_{n+\sigma}^{\alpha_{i_1}},
\end{align*}
so $(c)$ can be obtained by combining $(a)$.

From Lemma \ref{multi-coeff-property}, we know that ${\bf g}_k^{(n+1)}>{\bf g}_0^{(n+1)}$ for $1\leq k\leq n$. Then applying Lemma \ref{coefficient-pro1} $(a)$ again, we get
\begin{align*}
\tau\sum_{k=0}^n\frac1{{\bf g}_k^{(n+1)}} \leq \tau\sum_{k=0}^n\frac1{{\bf g}_0^{(n+1)}}\leq  C\tau\sum_{k=0}^{n} t_{n+\sigma}^{\alpha_{i_1}-1}
\leq C\tau^{\alpha_{i_1}}\sum_{k=0}^{n}(n+\sigma)^{\alpha_{i_1}-1}
\leq C\tau^{\alpha_{i_1}}(n+\sigma)^{\alpha_{i_1}}
= Ct_{n+\sigma}^{\alpha_{i_1}},
\end{align*}
 so (d) is proved.
\end{proof}

\subsection{The unconditional convergence}

For two mesh functions $v_h, w_h\in \mathcal{V}_{h} $, we define the inner product and norms (the discrete $L_2$-norm and a semi-norm)
$$\langle v,w\rangle=h\sum_{i=1}^{M-1}v_iw_i,\quad \|v\|=\sqrt{\langle v,v\rangle},\quad |v|_1=\|\delta_x v\|.$$
Furthermore, we introduce the discrete $H^1$-norm
$\|v\|_{H^1}=\sqrt{\|v\|^2+|v|_1^2}.$\\
Referring to \cite{Liao-NMPDE2015,Sun-Book}, we know that
$\|v\|\leq (x_R-x_L)/{\sqrt{6}}|v|_1$, then
$$\|v\|_{H^1}\leq\sqrt{1+\frac{(x_R-x_L)^2}{6}}|v|_1.$$
Now we denote the errors $e_i^n=u(x_i,t_n)-u_i^n$, $i\in {\bar I}$ and $0\leq n\leq N$,
and denote
\begin{align}\notag
&{^{\cal FH}}{\cal {\widehat D}}_t^\alpha e_i^{n+1}=\sum_{k=0}^n {\bf g}_{k}^{(n+1)}({\widehat e}_i^{k+1-\sigma}-{\widehat e}_i^{k-\sigma}),\\\label{error-discret-1}
&{^{\cal FH}}{\cal {\widehat D}}_t^\alpha e_i^{1}={\bf g}_0^{(1)}{\widehat e}_i^{1-\sigma},
\end{align}
in which
\begin{align}\notag
& {\widehat e}_i^{k+1-\sigma}=(2-2\sigma)\delta_te_i^{k+\frac12}+(2\sigma-1)\delta_{\hat t}e_i^k,\quad 1\leq k\leq n,\\\label{e1sigma}
&{\widehat e}_i^{1-\sigma}=(2-2\sigma)\frac{e_i^{1}}{\tau}, \quad
 {\widehat e}_i^{-\sigma}= {\tilde b_n}{\widehat e}_i^{1-\sigma}.
\end{align}
 Take
\begin{align}\notag
{\widehat w}_i^k=&(\frac32-\theta)\big[\theta e_i^{k}+(1-\theta)e_i^{k-1}\big]+(\theta-\frac12)\big[\theta
e_i^{k-1}+(1-\theta)e_i^{k-2}\big],\quad k\geq 2,\\\label{for-w1}
{\widehat w}_i^1=&[(\frac32-\theta)\theta+(\theta-\frac12)(1-\theta)] e_i^{1}.
\end{align}
From subsection \ref{scheme-section}, we easily obtain the following error system:
\begin{align}\label{err1}
&{^{\cal FH}}{\cal {\widehat D}}_t^\alpha e_i^{n+1}=\delta_x^2\Big(\frac{{\widehat w}_i^{n+1}+{\widehat w}_i^n}{2}\Big) +\big[f(u(x_i,t_n))-f(u_i^n)\big]+R_i^{n+1},
\quad 1\leq n\leq N-1, i\in I,\\\label{err2}
&{^{\cal FH}}{\cal {\widehat D}}_t^\alpha e_i^{1}={\tilde R}^\sigma_i+ R_i^1,\quad i\in I,\\\label{err3}
&e_0^n=e_M^n=0, \quad 1\leq n\leq N,\\\label{err4}
& e^0_i=0,\quad i\in {\bar I}.
\end{align}
Next theorem will show that the error $e_i^n$ is bounded in the $H^1$-norm unconditionally.

\begin{theorem}\label{convergence}
  Let $u(x,t)$ be the solution of the problem \eqref{eq1}--\eqref{eq3}. Assume $u(x,t)\in{\cal C}^4(\Omega)\cap {\cal C}^4[0,T]$. Let $\{u_i^n,i\in{\bar I},0\leq n\leq N\}$ be the solutions of the scheme \eqref{sch1}--\eqref{sch4}. Then the errors $e_i^n$ satisfy
\begin{align}\notag
\|e^n\|_{H^1}\leq C(\tau^2+h^2+\epsilon),\quad 0\leq n\leq N.
\end{align}
\end{theorem}

\begin{proof}
We finish the proof by using some techniques in Theorem 3.5 in \cite{VongLyu-JSC2018}.

From \eqref{err4}, one has $\|e^0\|_{\infty}=0$ .
We first utilize mathematical induction to show
\begin{align}\label{induction}
S^n\leq  C(\tau^2+h^2+\epsilon)^2
 + C\tau\sum_{k=0}^{n-1} \frac1{{\bf g}_0^{(k+1)}} S^k, \quad 1\leq n\leq N.
\end{align}
where
\begin{align}\label{sn}
S^n=\max\big\{(\|e^{n}\|+\sqrt{2C_\alpha}|\sigma e^{n}+(1-\sigma)e^{n-1}|_1)^2,  (2\sigma-1)^2C_\alpha|e^{n}|^2_1\big\}, ~ 1\leq n\leq N,~ \mbox{and} ~ S^0=0,
\end{align}
with $C_\alpha=Ct_{n+\sigma}^{2-\alpha_{i_1}}$.

It follows from \eqref{error-discret-1}--\eqref{e1sigma} that
\begin{align}\label{e1}
e_i^1=\frac{\tau}{(2-2\sigma){\bf g}_0^{(1)}}\left({\tilde R}^\sigma_i+ R_i^1 \right).
\end{align}
Then by \eqref{Rsigma} and Lemma \ref{coeff-property2} $(a)$, we have
\begin{align*}
|e^1|_1\leq C\tau^{\alpha_m}|{\tilde R}^\sigma |_1+C\tau^3 \leq C\tau^{3+\beta_m}+C\tau^3\leq C\tau^3.
\end{align*}
Hence \eqref{induction} holds for $n=0$ and $n=1$ ($\tau\leq1$).

Suppose \eqref{induction} is valid for $1\leq n\leq q$ ($1\leq q\leq N-1$), that is
\begin{align}\label{inductive-assumption}
S^n \leq  C(\tau^2+h^2+\epsilon)^2
 + C\tau\sum_{k=0}^{n-1} \frac1{{\bf g}_0^{(k+1)}} S^k, \quad 1\leq n\leq q.
\end{align}
Before proving that \eqref{induction} is valid for $n=q+1$, we show that the numerical solutions $u^n~(1\leq n\leq q)$ are uniformly bounded based on the inductive assumption. By using Lemma \ref{GRW} (Gronwall's inequality) and Lemma \ref{coeff-property2} $(c)$ on \eqref{inductive-assumption}, we can obtain
$$\|e^n\|_{\infty}^2\leq CS^n\leq  C(\tau^2+h^2+\epsilon)^2, \quad 1\leq n\leq q.$$
With the smooth assumption on the exact solution, which yields $ \|U^n\|_\infty\leq C_u$ for a positive constant $C_u$, it follows
$$\|u^n\|_\infty \leq \|U^n\|_\infty+\|e^n\|_\infty \leq C_u+1,\quad 1\leq n\leq q.$$
Hence we can take $\Omega_f=[-C_u-1,C_u+1]\cup [\inf\{\varphi+\sigma\tau\psi\},\sup\{\varphi+\sigma\tau\psi\}]$ in the rest proof.

We now verify that \eqref{induction} is valid for $n=q+1$.

Taking the inner product of \eqref{err1} with
\begin{align}\notag
 2\big[\sigma {\widehat e}_i^{n+1-\sigma} +(1-\sigma){\widehat e}_i^{n-\sigma}\big]=2\Big(\frac{{\widehat w}_i^{n+1}-{\widehat w}_i^n}{\tau}\Big),
 \quad 1\leq n\leq q,
\end{align}
we have
\begin{align}\label{error1_1}
2\left\langle {^{\cal FH}}{\cal {\widehat D}}_t^\alpha e^{n+1}-{\tilde F}^{n},\sigma {\widehat e}^{n+1-\sigma} +(1-\sigma){\widehat e}^{n-\sigma}\right\rangle
=2\left\langle \delta_x^2 \Big( \frac{{\widehat w}^{n+1}+{\widehat w}^n}{2} \Big),\frac{{\widehat w}^{n+1}-{\widehat w}^n}{\tau} \right\rangle  ,
\end{align}
where ${\tilde F}_i^{n}=\big[f(u(x_i,t_n))-f(u_i^n)\big] +R_i^{n+1}$.

With the boundary values being zero, it is easy to verify that
\begin{align}\label{wn1norm}
-2\left\langle \delta_x^2 \Big(\frac{{\widehat w}^{n+1}+{\widehat w}^n}{2}\Big),\frac{{\widehat w}^{n+1}-{\widehat w}^n}{\tau}
\right\rangle=\frac{|{\widehat w}^{n+1}|_1^2-|{\widehat w}^{n}|_1^2}{\tau}.
\end{align}
Utilizing Lemma \ref{stproof1}, we get
\begin{align}\notag
&2\left\langle {^{\cal FH}}{\cal {\widehat D}}_t^\alpha e^{n+1}-{\tilde F}^{n}, \sigma {\widehat e}^{n+1-\sigma}+(1-\sigma){\widehat e}^{n-\sigma}\right\rangle\\\label{wn1norm2}
\geq &\sum_{k=0}^n {\bf g}_{k}^{(n+1)} \|{\widehat e}^{k+1-\sigma}\|^2-\sum_{k=0}^{n-1} {\bf g}_{k}^{(n)}\|{\widehat e}^{k+1-\sigma}\|^2-\left({\bf g}_1^{(n+1)}-{\bf g}_0^{(n)} \right)\|{\widehat e}^{1-\sigma}\|^2
     -{\bf g}_0^{(n+1)}\Big\|{\widehat e}^{-\sigma}+\frac1{{\bf g}_0^{(n+1)}}{\tilde F}^{n} \Big\|^2.
\end{align}
Substituting \eqref{wn1norm} and \eqref{wn1norm2} into \eqref{error1_1}, we obtain
\begin{align}\label{error2}
 E^{n+1}-E^n \leq \tau \left|{\bf g}_1^{(n+1)}-{\bf g}_0^{(n)} \right| \|{\widehat e}^{1-\sigma}\|^2+
 \tau{\bf g}_0^{(n+1)}\|{\widehat e}^{-\sigma}\|^2+ \frac{\tau}{{\bf g}_0^{(n+1)}}\|{\tilde F}^{n}\|^2,\quad
1\leq n\leq p,
\end{align}
where
$$E^n=\tau \sum_{k=0}^{n-1} {\bf g}_{k}^{(n)}{\|{\widehat e}^{k+1-\sigma}\|}^2+|{\widehat w}^n|_1^2.$$
Summing up \eqref{error2} for $n$ from $1$ to $p$ yield
\begin{align*}
E^{p+1} \leq  E^{1}+\tau \sum_{n=1}^q \left|{\bf g}_1^{(n+1)}-{\bf g}_0^{(n)} \right| \|{\widehat e}^{1-\sigma}\|^2+ \tau\sum_{n=1}^q {\bf g}_0^{(n+1)} \|{\widehat e}^{-\sigma}\|^2 +\tau\sum_{n=1}^q\frac{1}{{\bf g}_0^{(n+1)}}\|{\tilde F}^{n}\|^2,
\end{align*}
that is
\begin{align}\nonumber
\tau \sum_{k=0}^{q} {\bf g}_{k}^{(q+1)} \|{\widehat e}^{k+1-\sigma}\|^2+|{\widehat w}^{q+1}|_1^2
\leq & |{\widehat w}^{1}|_1^2+\tau {\bf g}_0^{(1)} \|{\widehat e}^{1-\sigma}\|^2+\tau \sum_{n=1}^q \left|{\bf g}_1^{(n+1)}-{\bf g}_0^{(n)} \right| \|{\widehat e}^{1-\sigma}\|^2\\\label{sumck1}
&+ \tau\sum_{n=1}^q {\bf g}_0^{(n+1)} \|{\widehat e}^{-\sigma}\|^2 +\tau\sum_{n=1}^q\frac{1}{{\bf g}_0^{(n+1)}}\|{\tilde F}^{n}\|^2.
\end{align}
It can be verified by using Cauchy-Schwarz inequality and Lemma \ref{coeff-property2} $(d)$ that
\begin{align}\label{sumck2}
\Big\|\tau\sum_{k=0}^{q} {\widehat e}^{k+1-\sigma}\Big\|^2 \leq\left(\tau\sum_{k=0}^{q}\frac{1}{{\bf g}_{k}^{(q+1)}} \right) {\tau}\sum_{k=0}^{q}{\bf g}_{k}^{(q+1)} \|{\widehat e}^{k+1-\sigma}\|^2
\leq C_{\alpha}\left( {\tau}\sum_{k=0}^{q}{\bf g}_{k}^{(q+1)} \|{\widehat e}^{k+1-\sigma}\|^2\right).
\end{align}
Furthermore, the inequality $2(y^2+z^2)\geq (y+z)^2$ gives
\begin{align}\nonumber
2\left(\Big\|(\sigma-\frac12)e^1\Big\|^2+\Big\|\tau\sum_{k=0}^{q} {\widehat e}^{k+1-\sigma}\Big\|^2\right)\geq&
\Big\|(\sigma-\frac12)e^1+\tau\sum_{k=0}^{q} {\widehat e}^{k+1-\sigma}\Big\|^2\\\label{sumck3}
=&\Big\|(\frac{3}{2}-\sigma)e^{q+1}+(\sigma-\frac12)e^q\Big\|^2.
\end{align}
Consequently, it follows from \eqref{sumck1}--\eqref{sumck3} that
\begin{align}\label{errorB}
\Big\|(\frac{3}{2}-\sigma)e^{q+1}+(\sigma-\frac12)e^q\Big\|^2+2C_\alpha|{\widehat w}^{q+1}|_1^2\leq 2B^q,
\end{align}
where 
\begin{align}\notag
B^q=\Big\|(\sigma-\frac12)e^1\Big\|^2+C_\alpha\Big\{&|{\widehat w}^{1}|_1^2+\tau {\bf g}_0^{(n)} \|{\widehat e}^{1-\sigma}\|+\tau \sum_{n=1}^q \left|{\bf g}_1^{(n+1)}-{\bf g}_0^{(n)} \right| \|{\widehat e}^{1-\sigma}\|^2\\\label{Bm}
&+ \tau\sum_{n=1}^q {\bf g}_0^{(n+1)} \|{\widehat e}^{-\sigma}\|^2 +\tau\sum_{n=1}^q\frac{1}{{\bf g}_0^{(n+1)}}\|{\tilde F}^{n}\|^2\Big\}.
\end{align}
We further note that the term on the left hand side of \eqref{errorB} satisfies
\begin{align}\nonumber
&\Big\|(\frac{3}{2}-\sigma)e^{q+1}+(\sigma-\frac12)e^q\Big\|^2+2C_\alpha|{\widehat w}^{q+1}|_1^2\\\nonumber
\geq&\frac12\left(\Big \|(\frac{3}{2}-\sigma)e^{q+1}+(\sigma-\frac12)e^q\Big\|+\sqrt{2C_\alpha}|{\widehat w}^{q+1}|_1 \right)^2\\\nonumber
=& \frac12\left( \Big\|(\frac{3}{2}-\sigma)e^{q+1}+(\sigma-\frac12)e^q\Big\|\right.\\\nonumber
&\left.+\sqrt{2C_\alpha}\Big|(\frac32-\sigma)\big[\sigma e^{q+1}+(1-\sigma)e^q\big]+(\sigma-\frac12)\big[\sigma e^{q}+(1-\sigma)e^{q-1}\big]\Big|_1\right)^2\\\nonumber
\geq &\frac12\left[ (\frac{3}{2}-\sigma)\Big(\|e^{q+1}\|+\sqrt{2C_\alpha}\big|\sigma e^{q+1}+(1-\sigma)e^q\big|_1\Big)\right.\\\label{errorB1}
&\left.-(\sigma-\frac12)\Big(\|e^q\|+\sqrt{2C_\alpha}\big|\sigma e^{q}+(1-\sigma)e^{q-1}\big|_1\Big)\right]^2.
\end{align}
Combining \eqref{errorB} and \eqref{errorB1}, we get
{\small
\begin{align}\label{errorB2}
\left[ \Big(\frac{3}{2}-\sigma\Big)\Big(\|e^{q+1}\|+\sqrt{2C_\alpha}\big|\sigma e^{q+1}+(1-\sigma)e^q\big|_1\Big)-\Big(\sigma-\frac12\Big)\Big(\|e^q\|+\sqrt{2C_\alpha}\big|\sigma e^{q}+(1-\sigma)e^{q-1}\big|_1\Big)\right]^2
\leq 4B^q.
\end{align}}
Referring to the proof of Theorem 3.5 in \cite{VongLyu-JSC2018}, we discuss the following two cases:\\
{\bf Case (I)} $\|e^{q+1}\|+\sqrt{2C_\alpha}\big|\sigma e^{q+1}+(1-\sigma)e^q\big|_1\leq \|e^q\|+\sqrt{2C_\alpha}\big|\theta e^{q}+(1-\sigma)e^{q-1}\big|_1$.


Similar to that in \cite{VongLyu-JSC2018}, we can obtain $S^{q+1}\leq S^q$ in this case, so \eqref{induction} follows directly.\\
{\bf Case (II)} $\|e^{q+1}\|+\sqrt{2C_\alpha}\big|\sigma e^{q+1}+(1-\sigma)e^q\big|_1\geq \|e^q\|+\sqrt{2C_\alpha}\big|\sigma e^{q}+(1-\sigma)e^{q-1}\big|_1$.

In this situation, we have
\begin{align*}
&(\frac{3}{2}-\sigma)\Big(\|e^{q+1}\|+\sqrt{2C_\alpha}\big|\sigma e^{q+1}+(1-\sigma)e^q\big|_1\Big)-(\sigma-\frac12)\Big(\|e^q\|+\sqrt{2C_\alpha}\big|\sigma e^{q}+(1-\sigma)e^{q-1}\big|_1\Big)\\
&\geq (2-2\sigma)\Big( \|e^{q+1}\|+\sqrt{2C_\alpha}\big|\sigma e^{q+1}+(1-\sigma)e^q\big|_1\Big).
\end{align*}
With the above inequality and \eqref{errorB2}, it follows
\begin{align}\label{errorSituationII}
\Big( \|e^{q+1}\|+\sqrt{2C_\alpha}\big|\sigma e^{q+1}+(1-\sigma)e^q\big|_1\Big)^2\leq \frac{B^q}{(1-\sigma)^2}.
\end{align}
We next estimate $B^q$ term by term. Firstly, with \eqref{e1}, we have
\begin{align}\label{errore1}
2\|(\sigma-\frac12)e^1\|^2= 2(\sigma-\frac12)^2\|e^1\|^2 \leq C \tau^4.
\end{align}
Combining \eqref{for-w1} and \eqref{e1}, we get
\begin{align}\label{errorw1}
|{\widehat w}^1|_1^2&=(3\sigma-2\sigma^2-\frac12)^2|e^1|_1^2
\leq C \tau^4.
\end{align}
%
From \eqref{error-discret-1}, \eqref{err2} and \eqref{Rsigma}, we have
$${\widehat e}_i^{1-\sigma}=\frac{1}{{\bf g}_0^{(1)}}({\tilde R}_i^\sigma+R_i^1)\leq C\tau^2.$$
Then with Lemma \ref{coeff-property2} $(a)$,
\begin{align}\label{estimate-ehat1}
\tau {\bf g}_0^{(1)} \|{\widehat e}^{1-\sigma}\|^2\leq C\tau^4.
\end{align}
Observing \eqref{fast-coefficient} and \eqref{bfg}, we know that ${\bf g}_0^{(n)}={^{\cal F}{{\widehat g}_0^{n}}}-\frac12b_{n-1}$ for $n\geq 2$, ${\bf g}_0^{(1)}={^{\cal F}{{\widehat g}_0^{1}}}$,  and ${\bf g}_1^{(n+1)}={^{\cal F}{\widehat  g}_1^{n+1}}={^{\cal F}{\widehat g}}_0^n+b_n$ for $n\geq 1$. Then we get
\begin{align}\notag
\tau\dsum_{n=1}^q\left|{\bf g}_1^{(n+1)}-{\bf g}_0^{(n)}\right|=&\tau\left(\frac32\dsum_{n=1}^qb_n-\frac12b_q\right)
\leq \frac32\tau\dsum_{n=1}^qb_n
\leq C\tau\dsum_{n=1}^q{\bf g}_0^{(n+1)}\leq C,
\end{align}
where Lemma \ref{coefficient-pro1} (b) and Lemma \ref{coeff-property2} (b) have been used in the last two inequalities. Therefore,
\begin{align}\label{estimate-ehat2}
\tau \sum_{n=1}^q \left|{\bf g}_1^{(n+1)}-{\bf g}_0^{(n)} \right| \|{\widehat e}^{1-\sigma}\|^2\leq C\tau^4.
\end{align}
%
From \eqref{e1sigma} and Lemma \ref{coefficient-pro1} (b),
$$\|{\widehat e}^{-\sigma}\|={\tilde b_n}\|{\widehat e}^{1-\sigma}\|=\frac{(3\sigma-1)b_n}{2(1-\sigma){\bf g}_0^{(n+1)}}\|{\widehat e}^{1-\sigma}\|\leq C \|{\widehat e}^{1-\sigma}\|\leq C\tau^2,$$
thus it follows
\begin{align}\label{estimate-ehat2}
 \tau\sum_{n=1}^q {\bf g}_0^{(n+1)} \|{\widehat e}^{-\sigma}\|^2\leq C\tau^4.
\end{align}
Note that
\begin{align}\notag
\tau\sum_{n=1}^q\frac{1}{{\bf g}_0^{(n+1)}}\|{\tilde F}^{n}\|^2\leq &
\tau\sum_{n=1}^q\frac{1}{{\bf g}_0^{(n+1)}}\left(\|f(u(\cdot,t_n))-f(u^n)\|+\|R^{n+1}\| \right)^2\\\notag
\leq &
C\tau\sum_{n=1}^q\frac{1}{{\bf g}_0^{(n+1)}}\left(\|e^n\|+{\bf g}_0^{(n+1)}\tau^2+h^2+\epsilon \right)^2\\\notag
\leq &
C\tau\sum_{n=1}^q\frac{1}{{\bf g}_0^{(n+1)}}\|e^n\|^2+C\tau\sum_{n=1}^q{\bf g}_0^{(n+1)}\tau^4+C(h^2+\epsilon)^2\\\label{Fn}
\leq &
C\tau\sum_{n=0}^q\frac{1}{{\bf g}_0^{(n+1)}}S^n +C(\tau^2+h^2+\epsilon)^2,
\end{align}
where Lemma \ref{coeff-property2} $(b)$ and $(c)$ have been used.

Thus, \eqref{Bm} and \eqref{errorSituationII}--\eqref{Fn} yield
\begin{align}\label{smresult}
&\Big( \|e^{q+1}\|+\sqrt{2C_\alpha}\big|\sigma e^{q+1}+(1-\sigma)e^q\big|_1\Big)^2\leq
C\tau\sum_{n=0}^q\frac{1}{{\bf g}_0^{(n+1)}}S^n +C(\tau^2+h^2+\epsilon)^2.
\end{align}
Similar discussion with {\bf Case (II)}  in Theorem 3.5 of \cite{VongLyu-JSC2018}, we can get
\begin{align}\label{caseII}
S^{q+1}\leq S^{q}\quad \mbox{or} \quad
S^{q+1}=\Big( \|e^{q+1}\|+\sqrt{2C_\alpha}\big|\sigma e^{q+1}+(1-\sigma)e^q\big|_1\Big)^2.
\end{align}
Hence, the inequality \eqref{induction} is clarified according to \eqref{smresult}--\eqref{caseII}.

Consequently, we can apply Lemma \ref{GRW} (Gronwall's inequality) and Lemma \ref{coeff-property2} $(c)$ on \eqref{induction} to conclude
$$S^n\leq {C}(\tau^2+h^2+\epsilon)^2, \quad 0\leq n\leq N.$$
With \eqref{sn}, we then obtain the desired result.
\end{proof}

\begin{remark}
In a way similar to the proof of convergence, we can show that the numerical scheme \eqref{sch1}--\eqref{sch4}  is unconditionally stable with respect to discrete $H^1$-norm. Readers can refer to \cite{VongLyu-JSC2018} for more details.
\end{remark}

\section{Numerical experiments}\label{Sec-numer}

In this section, we carry out numerical experiments for the proposed finite difference schemes \eqref{sch1}-\eqref{sch4} to illustrate our theoretical statements.
The $H^1$-norm errors between the exact and the numerical solutions $$E_{1}(h,\tau)=\max_{0\leq n\leq N}\|e^n\|_{H^1}$$ are shown in the following tables and the convergence rates defined by
$$\mbox{Rate1}=\log_2\bigg[\dfrac{E_1(h,2\tau)}{E_1(h,\tau)}\bigg],~
 \mbox{Rate2}=\log_2\bigg[\dfrac{E_1(2h,\tau)}{E_1(h,\tau)}\bigg],$$
 are also recorded.

In the following tests, scheme1 stands for  the scheme \eqref{sch1}-\eqref{sch4}, and scheme2 represents the direct scheme which is similar to the scheme1 but without using the SOE.

We consider the problem for $x\in[0,1],~T=1$ and the forcing term
$$p(x,t)=\left[\dsum_{r=0}^m24\lambda_r\frac{t^{4-\alpha_r}}{\Gamma(4-\alpha_r)}+\pi^2t^4\right]\sin(\pi x)+f\big[\sin(\pi x)t^4\big],$$
is chosen to such the exact solution $u(x,t)=\sin(\pi x)t^4$, where
\begin{align*}
\textbf{Case1}~~&f(u)=2u^3,\\
\textbf{Case2}~~&f(u)=\sin(u),\\
\textbf{Case3}~~&f(u)=\left(u^2+5\right)^{\frac12}.
\end{align*}
\begin{table}[htbp]
 \begin{center}
 \caption{Numerical convergence orders of scheme \eqref{sch1}-\eqref{sch4} in temporal direction with $h=\frac{\pi}{1000}$ and $\epsilon_r=\tau^{4-\alpha_r}\times10^{-3}$.}\label{table1}
 \renewcommand{\arraystretch}{0.9}
 \def\temptablewidth{0.9\textwidth}
 {\rule{\temptablewidth}{0.7pt}}
 \begin{tabular*}{\temptablewidth}{@{\extracolsep{\fill}}lllllll}
&$(\alpha_0,\alpha_1,\alpha_2)$& $\tau$  &\multicolumn{2}{c}{$(\lambda_0,\lambda_1,\lambda_2)=(3,2,1)$}&\multicolumn{2}{c}{$(\lambda_0,\lambda_1,\lambda_2)=(1,2,3)$}\\
 \cline{4-5}\cline{6-7}
 &    &        &$E_1(h,\tau)$&Rate1&$E_1(h,\tau)$
 &Rate1\\\hline

 &(1.9,1.5,1.2) &$1/20$   & 2.7876e-03  & $\ast$  & 3.3012e-03  & $\ast$  \\
 &    &$1/40$  & 6.8270e-04  & 2.0297  & 7.9001e-04  & 2.0631  \\
 &    &$1/80$  & 1.6690e-04  & 2.0323  & 1.8917e-04  & 2.0622  \\
 &    &$1/160$  & 4.0829e-05  & 2.0313  & 4.5438e-05  & 2.0577  \\
\textbf{Case 1} &(1.8,1.4,1.3) &$1/20$   & 3.9224e-03  & $\ast$  & 3.7014e-03  & $\ast$  \\
 &    &$1/40$  & 7.8078e-04  & 2.0255  & 9.7427e-04  & 2.0093  \\
 &    &$1/80$  & 1.9145e-04  & 2.0279  & 2.4156e-04  & 2.0120  \\
 &    &$1/160$  & 4.6959e-05  & 2.0275  & 5.9816e-05  & 2.0138  \\
 &(1.6,1.5,1.2)&$1/20$   & 3.7121e-03  & $\ast$  & 3.9224e-03  & $\ast$  \\
 &    &$1/40$  & 9.2321e-04  & 2.0075  & 9.7427e-04  & 2.0093  \\
 &    &$1/80$  & 2.2903e-04  & 2.0111  & 2.4156e-04  & 2.0120  \\
 &    &$1/160$  & 5.6786e-05  & 2.0119  & 5.9816e-05  & 2.0138  \\\hline
 &(1.9,1.5,1.2) &$1/20$   & 2.7710e-03  & $\ast$  & 3.2879e-03  & $\ast$  \\
 &    &$1/40$  & 6.7866e-04   & 2.0297  & 7.8722e-04  & 2.0623  \\
 &    &$1/80$  & 1.6591e-04  & 2.0323  & 1.8854e-04  & 2.0619  \\
 &    &$1/160$  & 4.0587e-05  & 2.0313  & 4.5292e-05  & 2.0575  \\
\textbf{Case 2} &(1.8,1.4,1.3) &$1/20$   & 3.1645e-03  & $\ast$  & 3.6318e-03  & $\ast$  \\
 &    &$1/40$  & 7.7725e-04  & 2.0255  & 8.8483e-04  & 2.0372  \\
 &    &$1/80$  & 1.9049e-04  & 2.0286  & 2.1474e-04  & 2.0428  \\
 &    &$1/160$  & 4.6706e-05  & 2.0281  & 5.2092e-05   & 2.0434  \\
 &(1.6,1.5,1.2)&$1/20$   & 3.6983e-03   & $\ast$  & 3.9177e-03  & $\ast$  \\
 &    &$1/40$  & 9.2024e-04  & 2.0068  & 9.7389e-04  & 2.0082  \\
 &    &$1/80$  & 2.2833e-04  & 2.0109  & 2.4152e-04  & 2.0116  \\
 &    &$1/160$  & 5.6615e-05  & 2.0119  & 5.9812e-05  & 2.0136  \\ \hline
 &(1.9,1.5,1.2) &$1/20$   & 2.8107e-03  & $\ast$  & 3.3492e-03  & $\ast$  \\
 &    &$1/40$  & 6.8858e-04  & 2.0292  & 8.0192e-04  & 2.0623  \\
 &    &$1/80$  & 1.6834e-04  & 2.0322  & 1.9201e-04  & 2.0623  \\
 &    &$1/160$  & 4.1184e-05  & 2.0313  & 4.6092e-05  & 2.0586  \\
\textbf{Case 3} &(1.8,1.4,1.3) &$1/20$   & 3.2148e-03  & $\ast$  & 3.7014e-03  & $\ast$  \\
 &    &$1/40$  & 7.8977e-04  & 2.0252  & 9.0182e-04  & 2.0372  \\
 &    &$1/80$  & 1.9357e-04  & 2.0286  & 2.1884e-04  & 2.0430  \\
 &    &$1/160$  & 4.7458e-05  & 2.0281  & 5.3082e-05  & 2.0436  \\
 &(1.6,1.5,1.2)&$1/20$   & 3.7646e-03  & $\ast$  & 3.9971e-03  & $\ast$  \\
 &    &$1/40$  & 9.3687e-04  & 2.0066  & 9.9369e-04  & 2.0081  \\
 &    &$1/80$  & 2.3246e-04  & 2.0109  & 2.4642e-04  & 2.0117  \\
 &    &$1/160$  & 5.7637e-05  & 2.0119  & 6.1024e-05  & 2.0137  \\
\end{tabular*}
{\rule{\temptablewidth}{0.7pt}}

\end{center}
\end{table}

\begin{table}[htbp]
 \begin{center}
 \caption{Numerical convergence orders of direct scheme (scheme2) in temporal direction with $h=\frac{\pi}{1000}$ .}\label{table2}
 \renewcommand{\arraystretch}{0.9}
 \def\temptablewidth{0.9\textwidth}
 {\rule{\temptablewidth}{0.7pt}}
 \begin{tabular*}{\temptablewidth}{@{\extracolsep{\fill}}lllllll}
&$(\alpha_0,\alpha_1,\alpha_2)$& $\tau$  &\multicolumn{2}{c}{$(\lambda_0,\lambda_1,\lambda_2)=(3,2,1)$}&\multicolumn{2}{c}{$(\lambda_0,\lambda_1,\lambda_2)=(1,2,3)$}\\
 \cline{4-5}\cline{6-7}
 &    &        &$E_1(h,\tau)$&Rate1&$E_1(h,\tau)$
 &Rate1\\\hline

 &(1.9,1.5,1.2) &$1/20$   & 2.7876e-03  & $\ast$  & 3.3012e-03  & $\ast$  \\
 &    &$1/40$  & 6.8271e-04  & 2.0297  & 7.9001e-04  & 2.0631  \\
 &    &$1/80$  & 1.6689e-04  & 2.0324  & 1.8916e-04  & 2.0623  \\
 &    &$1/160$  & 4.0830e-05  & 2.0312  & 4.5463e-05  & 2.0568  \\
\textbf{Case 1} &(1.8,1.4,1.3) &$1/20$   & 3.1816e-03  & $\ast$  & 3.6421e-03  & $\ast$  \\
 &    &$1/40$  & 7.8130e-04  & 2.0258  & 8.8677e-04  & 2.0381  \\
 &    &$1/80$  & 1.9147e-04  & 2.0287  & 2.1516e-04  & 2.0431  \\
 &    &$1/160$  & 4.6950e-05  & 2.0280  & 5.2184e-05  & 2.0438  \\
 &(1.6,1.5,1.2)&$1/20$   & 3.7121e-03  & $\ast$  & 3.9224e-03  & $\ast$  \\
 &    &$1/40$  & 9.2322e-04  & 2.0075  & 9.7428e-04  & 2.0093  \\
 &    &$1/80$  & 2.2903e-04  & 2.0112  & 2.4155e-04  & 2.0120  \\
 &    &$1/160$  & 5.6776e-05  & 2.0122  & 5.9846e-05  & 2.0130  \\\hline
 &(1.9,1.5,1.2) &$1/20$   & 2.7711e-03  & $\ast$  & 3.2880e-03  & $\ast$  \\
 &    &$1/40$  & 6.7867e-04  & 2.0297  & 7.8723e-04  & 2.0623  \\
 &    &$1/80$  & 1.6590e-04  & 2.0324  & 1.8853e-04  & 2.0620  \\
 &    &$1/160$  & 4.0588e-05  & 2.0312  & 4.5316e-05  & 2.0567  \\
\textbf{Case 2} &(1.8,1.4,1.3) &$1/20$   & 3.1645e-03  & $\ast$  & 3.6318e-03  & $\ast$  \\
 &    &$1/40$  & 7.7726e-04  & 2.0255  & 8.8485e-04  & 2.0372  \\
 &    &$1/80$  & 1.9050e-04  & 2.0286  & 2.1474e-04  & 2.0428  \\
 &    &$1/160$  & 4.6712e-05  & 2.0279  & 5.2088e-05  & 2.0436  \\
 &(1.6,1.5,1.2)&$1/20$   & 3.6983e-03  & $\ast$  & 3.9177e-03  & $\ast$  \\
 &    &$1/40$  & 9.2025e-04  & 2.0068  & 9.7391e-04  & 2.0082  \\
 &    &$1/80$  & 2.2832e-04  & 2.0109  & 2.4152e-04  & 2.0117  \\
 &    &$1/160$  & 5.6605e-05  & 2.0121  & 5.9842e-05  & 2.0129  \\ \hline
 &(1.9,1.5,1.2) &$1/20$   & 2.8107e-03  & $\ast$  & 3.3493e-03  & $\ast$  \\
 &    &$1/40$  & 6.8858e-04  & 2.0292  & 8.0194e-04  & 2.0623  \\
 &    &$1/80$  & 1.6834e-04  & 2.0323  & 1.9203e-04  & 2.0621  \\
 &    &$1/160$  & 4.1184e-05  & 2.0312  & 4.6154e-05  & 2.0568  \\
\textbf{Case 3} &(1.8,1.4,1.3) &$1/20$   & 3.2148e-03  & $\ast$  & 3.7015e-03  & $\ast$  \\
 &    &$1/40$  & 7.8978e-04  & 2.0252  & 9.0183e-04  & 2.0372  \\
 &    &$1/80$  & 1.9357e-04  & 2.0286  & 2.1885e-04  & 2.0429  \\
 &    &$1/160$  & 4.7464e-05  & 2.0279  & 5.3078e-05   & 2.0437  \\
 &(1.6,1.5,1.2)&$1/20$   & 3.7647e-03  & $\ast$  & 3.9972e-03  & $\ast$  \\
 &    &$1/40$  & 9.3688e-04  & 2.0066  & 9.9370e-04  & 2.0081  \\
 &    &$1/80$  & 2.3245e-04  & 2.0109  & 2.4642e-04  & 2.0117  \\
 &    &$1/160$  & 5.7627e-05  & 2.0121  & 6.1054e-05  & 2.0129  \\
\end{tabular*}
{\rule{\temptablewidth}{0.7pt}}

\end{center}
\end{table}

\begin{table}[htbp]
 \begin{center}
 \caption{Numerical convergence orders of \eqref{sch1}-\eqref{sch4} scheme in spatial direction with $\tau=\frac{1}{1000}$ and $\epsilon_r=\tau^{4-\alpha_r}\times10^{-3}.$}\label{table3}
 \renewcommand{\arraystretch}{0.9}
 \def\temptablewidth{0.9\textwidth}
 {\rule{\temptablewidth}{0.7pt}}
 \begin{tabular*}{\temptablewidth}{@{\extracolsep{\fill}}lllllll}
&$(\alpha_0,\alpha_1,\alpha_2)$& $h$  &\multicolumn{2}{c}{$(\lambda_0,\lambda_1,\lambda_2)=(3,2,1)$}&\multicolumn{2}{c}{$(\lambda_0,\lambda_1,\lambda_2)=(1,2,3)$}\\
 \cline{4-5}\cline{6-7}
 &    &        &$E_1(h,\tau)$&Rate2&$E_1(h,\tau)$
 &Rate2\\\hline

 &(1.9,1.5,1.2) &$1/10$   &  7.1988e-04  & $\ast$  & 1.0492e-03  & $\ast$  \\
 &    &$1/20$  & 1.5600e-04  & 2.2062 & 2.2858e-04  & 2.1985  \\
 &    &$1/40$  & 3.5095e-05  & 2.1523  & 5.2581e-05  & 2.1201  \\
 &    &$1/80$  & 7.7903e-06  & 2.1715  & 1.2792e-05  & 2.0393  \\
\textbf{Case 1} &(1.8,1.4,1.3) &$1/10$   & 8.2721e-04  & $\ast$  & 1.0947e-03  & $\ast$  \\
 &    &$1/20$  & 1.7903e-04  & 2.2081  & 2.3707e-04  & 2.2071  \\
 &    &$1/40$  & 4.0068e-05  & 2.1597  & 5.3245e-05  & 2.1546  \\
 &    &$1/80$  & 8.6959e-06  & 2.2040  & 1.1739e-05   & 2.1814  \\
 &(1.6,1.5,1.2)&$1/10$   & 9.8274e-04  & $\ast$  & 1.1904e-03  & $\ast$  \\
 &    &$1/20$  & 2.1253e-04  & 2.2091  & 2.5784e-04  & 2.2069  \\
 &    &$1/40$  & 4.7440e-05  & 2.1635  & 5.7963e-05  & 2.1533  \\
 &    &$1/80$  & 1.0175e-05  & 2.2210  & 1.2831e-05  & 2.1755  \\\hline
 &(1.9,1.5,1.2) &$1/10$   & 7.2106e-04  & $\ast$  & 1.0536e-03  & $\ast$  \\
 &    &$1/20$  & 1.5627e-04  & 2.2061  & 2.2955e-04  & 2.1984  \\
 &    &$1/40$  & 3.5160e-05  & 2.1520  & 5.2809e-05 & 2.1199  \\
 &    &$1/80$  & 7.8106e-06  & 2.1704  & 1.2853e-05  & 2.0387  \\
\textbf{Case 2} &(1.8,1.4,1.3) &$1/10$   & 8.2933e-04  & $\ast$  & 1.0536e-03  & $\ast$  \\
 &    &$1/20$  & 1.7950e-04  & 2.2080  & 2.2955e-04  & 2.1984  \\
 &    &$1/40$  & 4.0180e-05  & 2.1594  & 5.2809e-05  & 2.1199  \\
 &    &$1/80$  & 8.7272e-06  & 2.2029  & 1.2853e-05  & 2.0387  \\
 &(1.6,1.5,1.2)&$1/10$   & 9.8719e-04  & $\ast$  & 1.1988e-03  & $\ast$  \\
 &    &$1/20$  & 2.1350e-04  & 2.2091  & 2.5968e-04 & 2.2068  \\
 &    &$1/40$  & 4.7666e-05  & 2.1632  & 5.8386e-05  & 2.1531  \\
 &    &$1/80$  & 1.0232e-05  & 2.2198  & 1.2934e-05  & 2.1745  \\ \hline
 &(1.9,1.5,1.2) &$1/10$   & 7.2403e-04  & $\ast$  & 1.0605e-03  & $\ast$  \\
 &    &$1/20$  & 1.5690e-04  & 2.2062 & 2.3104e-04  & 2.1985  \\
 &    &$1/40$  & 3.5295e-05  & 2.1523  & 5.3145e-05  & 2.1202  \\
 &    &$1/80$  & 7.8325e-06  & 2.1719  & 1.2928e-05  & 2.0395  \\
\textbf{Case 3} &(1.8,1.4,1.3) &$1/10$   & 8.3342e-04  & $\ast$  & 1.0947e-03  & $\ast$  \\
 &    &$1/20$  & 1.8037e-04  & 2.2081  & 2.4001e-04  & 2.2071  \\
 &    &$1/40$  & 4.0365e-05  & 2.1598  & 5.3899e-05  & 2.1547  \\
 &    &$1/80$  & 8.7573e-06  & 2.2046  & 1.1878e-05   & 2.1819  \\
 &(1.6,1.5,1.2)&$1/10$   & 9.9324e-04  & $\ast$  & 1.2081e-03  & $\ast$  \\
 &    &$1/20$  & 2.1479e-04  & 2.2092  & 2.6167e-04  & 2.2069  \\
 &    &$1/40$  & 4.7941e-05  & 2.1636  & 5.8817e-05  & 2.1534  \\
 &    &$1/80$  & 1.0278e-05  & 2.2217  & 1.3015e-05  & 2.1761  \\
\end{tabular*}
{\rule{\temptablewidth}{0.7pt}}

\end{center}
\end{table}

\begin{table}[htbp]
 \begin{center}
 \caption{CPU in seconds of scheme1 and scheme2 with $h=\frac{\pi} {50}$, $\alpha=(1.6,1.5,1.2)$, $\lambda=(3,2,1)$ and $\epsilon_r=\tau^{4-\alpha_r}.$}\label{table4}
 \renewcommand{\arraystretch}{0.8}
 \def\temptablewidth{0.8\textwidth}
 {\rule{\temptablewidth}{0.7pt}}
 \begin{tabular*}{\temptablewidth}{@{\extracolsep{\fill}}lllllll}
   $\tau$   &\multicolumn{2}{c}{\textbf{Case 1}}&\multicolumn{2}{c}{\textbf{Case 2}}&\multicolumn{2}{c}{\textbf{Case 3}}\\
\cline{2-3}\cline{4-5}\cline{6-7}
             &scheme1&scheme2&scheme1&scheme2&scheme1&scheme2 \\\hline
$1/10000$    &1.40   &39.84  &1.33   &40.34  &1.34   &36.54 \\
$1/20000$    &5.00   &173.30 &4.58   &170.11 &4.59  &173.44 \\
$1/30000$    &18.52  &394.21 &13.21  &389.22 &12.83  &393.53 \\
$1/40000$    &40.18  &684.65 &42.28  &692.23 &39.69 &700.94\\\hline
\end{tabular*}
{\rule{\temptablewidth}{0.7pt}}
\end{center}
\end{table}
\begin{figure}[t]
  \centering
  \caption{The comparison of memory between scheme1 and scheme2 with $h=\frac{\pi} {50}$, $\alpha=(1.6,1.5,1.2)$, $\lambda=(3,2,1)$ and $\epsilon_r=\tau^{4-\alpha_r}.$}\label{fg1}
  \includegraphics[scale=0.65]{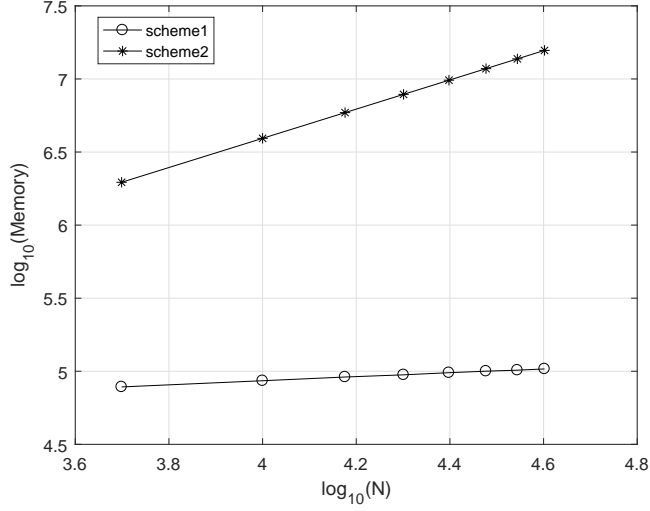}
\end{figure}

In Table \ref{table1}, taking serval sets of $\lambda$ and $\alpha$, with three different modeling cases, the $E_1(h,\tau)$ and the temporal convergence of scheme \eqref{sch1}-\eqref{sch4} are presented, which confirms the second-order convergence of the difference scheme with respect to temporal direction. 
In Table \ref{table2}, the numerical results are shown in temporal direction where the numerical results of three cases are also demonstrated and one can realize that scheme2 is of second-order convergence.
For the spatial direction, in Table \ref{table3}, the second-order accuracy of scheme \eqref{sch1}-\eqref{sch4} for three cases are also shown with fixed $\tau=\frac1{1000}.$ Above all, the chosen $\epsilon_r$ are $\tau^{4-\alpha_r}\times10^{-3}$, $r=0,1,2.$
Table \ref{table4} demonstrates CPU time in seconds of scheme1 and scheme2. Figure \ref{fg1}  presents the comparison between two schemes about the memory in bytes with different temporal step $\tau=\frac1{10000}$ to $\frac1{40000}$, and actually the results of three cases are the same. One can check that the scheme1 do require less time and memory and if $\tau$ is small enough, the comparison can be more evident.

\section{Conclusion}\label{conclusion}

We considered a fast and linearized finite difference method for solving the nonlinear multi-term time-fractional wave equation. The proposed scheme based on the fast ${\cal L}2$-$1_\sigma$ discretization, the multi-term ${\cal L}2$-$1_\sigma$ type discretization and a weighted approach. By showing some important properties of the refined coefficients of fully discretization, we obtained the truncation error of our proposed weighted discretization to the multi-term Caputo derivative, and we displayed the unconditional convergence  rigorously. The accuracy and efficiency of proposed method are well demonstrated by several numerical tests.


\end{document}